\documentclass[Final]{siamltex}
\usepackage[margin=1.2in]{geometry}
\usepackage{amsfonts}
\usepackage{mathrsfs}
\usepackage{amsmath}
\usepackage{amssymb}
\usepackage{amsmath}
\usepackage{float}
\usepackage{booktabs}
\usepackage{multirow}
\usepackage{ifpdf}
\usepackage{enumerate}
\usepackage{graphicx}
\usepackage{color}
\usepackage{lineno}

\usepackage{caption}
\usepackage[justification=centering]{caption}
\usepackage{bm}

\makeatletter

\newcommand{\Rmnum}[1]{\expandafter\@slowromancap\romannumeral #1@}
\makeatother

\newtheorem{assumption}[theorem]{Assumption}
\newtheorem{algorithm}[theorem]{Algorithm}

\usepackage{xcolor}	 % Required for custom colors
\definecolor{mygreen}{RGB}{44,85,17}
\definecolor{myblue}{RGB}{34,31,217}
\definecolor{mybrown}{RGB}{194,164,113}
\definecolor{myred}{RGB}{255,66,56}
\definecolor{mypurple}{RGB}{255 250 205}
\definecolor{myjackieblue}{RGB}{11,23,70}
\definecolor{mygrey}{RGB}{230 230 250}
\definecolor{rot}{rgb}{0.000,0.000,0.000}
\newcommand{\tcr}{\textcolor{rot}}
\definecolor{myblue}{rgb}{0.000,1.000,0.000}

\begin{document}
%%%%%\linenumbers

\title{Bayesian Approach to Inverse \tcr{Time-harmonic Acoustic Scattering with Phaseless Far-field Data}}

% my author command, not the siam uq requires

\author{Zhipeng Yang\footnotemark[3]\ , Xinping Gui\footnotemark[3]\ , Ju Ming\footnotemark[2] \and Guanghui Hu\footnotemark[3]\ }

\date{}
\maketitle

\renewcommand{\thefootnote}{\fnsymbol{footnote}}
\footnotetext[2]{School of Mathematics and Statistics, Huazhong University of Science and Technology, Wuhan 430074, P. R. China, (jming@hust.edu.cn).}
\footnotetext[3]{Department of Applied Mathematics, Beijing Computational Science Research Center, Beijing 100193, P. R. China, (yangzhp@csrc.ac.cn), (gui@csrc.ac.cn), (hu@csrc.ac.cn, corresponding author).}
\renewcommand{\thefootnote}{\arabic{footnote}}

\begin{abstract}
This paper is concerned with inverse acoustic scattering problem of inferring the position and  shape of a sound-soft obstacle from phaseless far-field data. We propose the Bayesian approach to recover sound-soft disks\tcr{, line cracks and kite-shaped obstacles} through properly chosen incoming waves in two dimensions. Given the Gaussian prior measure, the well-posedness of the posterior measure in the Bayesian approach is discussed. The Markov Chain Monte Carlo (MCMC) method is adopted in the numerical approximation and the preconditioned Crank-Nicolson (pCN) algorithm with random proposal variance is utilized to improve the convergence rate. Numerical examples are provided to illustrate effectiveness of the proposed method.
\end{abstract}



%%%%%%%%%%%%%%%%%
\section{Introduction}

Time-harmonic inverse scattering problems have attracted extensive attention due to their numerous applications in many areas such as radar and sonar detection, geophysical prospection, medical imaging, nondestructive testing and so on. In this paper, we are interested in the inverse problem of reconstructing the location and shape of an acoustically sound-soft obstacle using phaseless far-field data.

The propagation of a time-harmonic incident field $u^{in}$ in a homogeneous and isotropic medium is governed by the Helmholtz equation
\begin{equation}
    \Delta u^{in} + k^{2}u^{in} = 0 \hspace{.2 cm} \text{in} \hspace{.2 cm} \mathbb{R}^{2},
\end{equation}
where $k>0$ is the wavenumber. Let $D\subset \mathbb{R}^2$ be a sound-soft scatterer, which occupies a bounded subset with $C^{2}$-smooth boundary $\partial D$ such that the exterior $\mathbb{R}^{2} \backslash \bar{D}$ of $D$ is connected. In this paper $D$ maybe a domain or a curve, which represents an extended obstacle or a crack in acoustics.
The forward scattering problem is to find the scattered (perturbed) field $u^{sc}$ to the Helmholtz equation
\begin{equation}
    \Delta u^{sc} + k^{2}u^{sc} = 0 \hspace{.2 cm} \text{in} \hspace{.2 cm} \mathbb{R}^{2}\backslash \bar{D},
    \label{Helmholtz_eq}
\end{equation}
which satisfies the Dirichlet boundary condition
\begin{equation}
    u^{sc} =  -u^{in} \hspace{.2 cm} \text{on} \hspace{.2 cm} \partial D,
    \label{DiriCond_Helmholtz_eq}
\end{equation}
and the Sommerfeld radiation condition
\begin{equation}
    \lim_{r\rightarrow \infty} \sqrt{r} \left( \frac{\partial u^{sc}}{\partial r} - iku^{sc} \right) = 0, \hspace{.2 cm} r = |x|,
    \label{SommerfeldCond}
\end{equation}
uniformly in all directions $\hat{\mathbf{x}} = x/|x|\in \mathbb{S}:=\{x: |x|=1\}$, $x \in \mathbb{R}^{2} \backslash \bar{D}$\tcr{, and $i = \sqrt{-1}$ is the imaginary unit}. The total field $u$ is defined as $u = u^{in} + u^{sc}$ in $\mathbb{R}^{2}\backslash \bar{D}$. The Sommerfeld radiating solution $u^{sc}$ has an asymptotic behavior of the form
\begin{equation}
    u^{sc}(x) = \frac{ e^{\mathrm{i}k|x|} }{ \sqrt{|x|} } \left\{   u^{\infty}(\hat{\mathbf{x}}) + \mathcal{O}\left(\frac{ 1 }{ \sqrt{|x|} }\right)  \right\}, \hspace{.2 cm} |x| \rightarrow \infty,
    \label{asymptotic_behaviour}
\end{equation}
where $u^{\infty}(\hat{\mathbf{x}})$ is called the far-field pattern at the observation direction $\hat{\mathbf{x}}\in\mathbb{S}$. Note that $u^{\infty}: \mathbb{S}\rightarrow \mathbb{C}$ is an analytic function with phase information. The above model also appears in the TE polarization of time-harmonic electromagnetic scattering from infinitely long and perfectly conducting cylinders.

The uniqueness, stability and inversion algorithms for recovering $\partial D$ from phased far-field patterns have been extensively studied with one or many incoming plane and point source waves. We refer to the monographs \cite{CC2014, DColton_2013_InverseAcousticScattering, Akirsch_2008, Akirsch_2011_MathematicalTheoryIOP, GNakamura_RPotthast_2015_InverseModeling} for historical remarks, an overview of recent progresses and the comparison between different approaches. In many practical applications, the phase information of the far-field pattern cannot be measured accurately compared with its modulus or intensity. For instance, in optics it is not trivial to measure the phase of electromagnetic waves incited at high frequencies. One of the essential difficulties in using phaseless far-field data lies in the translation invariance property for plane wave incidence, which we state as follows. Let $u^\infty(\hat{\mathbf{x}}; D, \mathbf{d})$ be the far-field pattern corresponding to the incident plane wave $e^{ikx\cdot \mathbf{d}}$ ($\mathbf{d}\in \mathbb{S}$ is the incident direction) and the sound-soft obstacle $D$. For the shifted obstacle $D_z:=\{x+z: x\in D\}$, the corresponding far-field pattern is given by (see \cite{RKress_1997_ScatterCrackModulus})
\begin{equation}\label{translation}
    u^\infty(\hat{\mathbf{x}}; D_z, \mathbf{d})
    = e^{ik z\cdot(\mathbf{d}-\hat{\mathbf{x}})}u^\infty(\hat{\mathbf{x}}; D, \mathbf{d}) \quad\mbox{for all}\quad \hat{\mathbf{x}}\in\mathbb{S}.
\end{equation}
Hence, we get
\[
    |u^\infty(\hat{\mathbf{x}}; D_z, \mathbf{d})|
    = |u^\infty(\hat{\mathbf{x}}; D, \mathbf{d})|
    \quad\mbox{for all}     \quad
    \mathbf{d}, \hat{\mathbf{x}}\in\mathbb{S}.
\]
This implies that it is impossible to \tcr{recover} the location of $D$ from the phaseless far-field pattern of a plane wave.

There has been tremendous interest in inverse scattering with phaseless data or in phase retrieval problems in optics and other physical and engineering areas (see, e.g. \cite{HAmmari_2016_ScatteringPhaseless, GBao_2013_ScatteringPhaseless, FMMa_2018_ScatterCrackPhaseless, OIvanyshyn_2007_ShapeResconstruction, OIvanyshyn_2010_3DObstaclePhaseless, MVKlibanov_2014_3DScatterPhaseless, MVKlibanov_2017_3DScatterPhaseless, BZhang_2010_UniqueBallSingleFar, MHMaleki_1993_PhaseReconstruction} and the references therein). In a deterministic setting where randomness are not taken into account, Kress \& Rundell and Ivanyshyn \& Kress proposed a Newton-type iterative approach to reconstruct the shape of sound-soft obstacles from only the modulus of the far-field pattern in \cite{OIvanyshyn_2007_ShapeResconstruction, OIvanyshyn_2010_3DObstaclePhaseless, RKress_1997_ScatterCrackModulus}. The approach of \cite{OIvanyshyn_2007_ShapeResconstruction, OIvanyshyn_2010_3DObstaclePhaseless} was based on a pair of nonlinear and ill-posed integral equations motivated by an inverse boundary value problem for the Laplace equation with phase information \cite{IKress2006, RKress2005}; see also \cite{FMMa_2018_ScatterCrackPhaseless, OIvanyshyn_2008_InverseScatterCrackNonlinearIntegral, Kress1995_ScatterOpenArc} for inverse scattering from sound-soft cracks using a single far-field pattern with phase or phaseless information. Klibanov proved unique determination of a compactly supported potential of the stationary three-dimensional Schr$\ddot{\text{o}}$dinger equation from the phaseless near-field data incited by an interval of frequencies \cite{MVKlibanov_2014_3DScatterPhaseless}. This was later extended in \cite{MVKlibanov_2017_3DScatterPhaseless} to the reconstruction of a smooth wave speed in the three-dimensional Helmholtz equation. To broke the translation invariance property, it was recently prosed in \cite{BZhang_2018_UniquenessScatterPhaseless} that, phaseless far-field patterns generated by infinitely many sets of superpositions of two plane waves with different directions can be used to uniquely determine a penetrable or impenetrable scatterer; \tcr{see \cite{BZhang_HZhang_2018_FastImagingPhaselessFixedFrequency} for a fast imaging algorithm based on this idea.}
\tcr{Similar uniqueness results were derived by Zhang \& Guo \cite{DZhang_YGuo_2018_UniquePhaselessBall} where
the superposition of a fixed plane wave and some point sources was taken as incident waves and a reference ball technique was proposed.
Uniqueness and direct sampling algorithms using
the superposition of plane waves and fixed source location point sources were considered in \cite{XJi_XLiu_BZhang_2019_Phaseless_Unique_sampling}.} In this paper we propose to  generate the phaseless data using the following superposition of two plane waves (see \cite{BZhang_2018_UniquenessScatterPhaseless})
\begin{equation}
    u^{in}_{\ell}(x) := e^{ \mathrm{i} k x\cdot \mathbf{d}_{0} } + e^{\mathrm{i} k x\cdot \mathbf{d}_{\ell} }, \hspace{.2cm} \ell = 0,1, 2, \cdots, L, \label{incident_wave}
\end{equation}
and then to recover a sound-soft disk\tcr{, a line crack or a kite-shaped obstacle} through the Bayesian approach. In (\ref{incident_wave}), we fix $\mathbf{d}_{0} \in \mathbb{S}$ and change $\mathbf{d}_{\ell}\in\mathbb{S}$ as incident directions, due to the a priori information of the obstacle; see Theorem \ref{TH} for a uniqueness proof for sound-soft disks.

In recent years, the Bayesian method has received increasing attention for inverse problems \cite{BGFitzpatrick_1991_Bayesian, Stuart_2014_SubsurfaceFlow, JKaipio_ESomersalo_2006_StatisticalInverse, Stuart_2015_Data_Assimilation, Stuart_2010_Bayesian_perspective}, which also has been applied to the inverse scattering problems \cite{ABaussard_2001_BayesianScattringMicrowave, BThanh_GOmar_2014_BayesianScattering, IHarris_2017_NearField, JGSun_2019_ScatterBayesian, JGSun_2019_StekloffBayesian, YJWang_2015_BayesianScatteringInterior} with phase far-field data.
In particular, the authors of \cite{BThanh_GOmar_2014_BayesianScattering} adopt the Bayesian framework of \cite{Stuart_2010_Bayesian_perspective} to shape identification problems in inverse scattering and establish a framework for proving well-posedness of the Bayesian formulation using a suitable shape parametrization and the regularity of shape derivatives. The aim of this paper is to propose the Bayesian method using a single far-field pattern without phase information. The Bayesian method provides us a new perspective to view the inverse scattering problem in the form of statistical inferences. In this statistical approach, all parameters are random variables and the key issue is to estimate the posterior distribution of the unknown quantities based on the Bayes' formula \cite{Stuart_2010_Bayesian_perspective} and the known prior distribution.
\tcr{The Bayesian method could be an alternative method to overcome the challenges in deterministic inverse problems, although it usually leads to expensive computational cost.
The advantageous over
deterministic inversion schemes in inverse scattering (for example, optimization-based iterative schemes and non-iterative sampling methods) are summarized as follows.}
\tcr{(i) Instead of deterministic reconstructions, the Bayesian approach gives rise to statistic characteristics of the posterior distribution of unknown parameters and provides a quantification of the uncertainties arising from the corresponding model predictions.  (ii) The Bayesian method could lead to all possible solutions of the inverse problem. For example, using one plan wave rather than the superposition of two plane waves, the Bayesian method could reconstruct the shape of an unknown obstacle which is located at every possible position due to the translation invariance (\ref{translation}).
(iii) For inverse scattering problems, the theoretical analysis and numerical methods in the Bayesian framework are only based on the deterministic forward model. Hence, it is easy to perform theoretical analysis and numerical examples.
%\tcr{Compared with the deterministic approaches, which only gain one deterministic reconstruction, the statistic characteristics of this posterior distribution are used to estimate unknown parameters and provide a quantification of the uncertainties arising from the corresponding model predictions.
%Furthermore, when the incident wave is one plan wave rather than the superposition of two plane waves, the Bayesian method still could reconstruct the shape of the obstacle with all possible position.
%Moreover, the theoretical analysis and numerical methods in the Bayesian framework are only based on the deterministic forward model. Then compared with the deterministic approaches, it will be easier to do the theoretical analysis and numerical performances for the inverse scattering problem.}
(iv) The Bayesian method needs less measurement data without phase information and does not require a good initial guess.
In this paper, the Markov chain Monte Carlo (MCMC) method \cite{Brooks_2011_MCMC, Gamerman_2006_MCMC, Geyer_1992_MCMC} is proposed to accomplish the characterization of the posterior distribution, while the preconditioned Crank-Nicolson (pCN) algorithm \cite{Stuart_2013_pCN} is adopted to improve the convergence rate in the iteration of MCMC method. Since the MCMC method and the pCN algorithm are adopt to calculate the numerical approximation, the numerical method in this paper is insensitive to the initial guess of the obstacle shape. %While a suitable initial guess is the key point in the deterministic iterative approaches \cite{FMMa_2018_ScatterCrackPhaseless, OIvanyshyn_2007_ShapeResconstruction, OIvanyshyn_2010_3DObstaclePhaseless, RKress_1997_ScatterCrackModulus}. \tcb{[[[[[[ I do not know, whether these references added here are suitable or not.]]]]} The deterministic sampling approaches \tcb{[[[[[[ Maybe, some references are needed at here.]]]]} do not require a suitable initial guess, but the deterministic sampling approaches depend on a larger number of the observation data of the far-field pattern.
In our numerical examples, we exhibit that accurate reconstructions can be achieved when the number of incident waves and observation directions is small. %In summary, the Bayesian method could also deal with the deterministic inverse problems, when we can not give a suitable initial guess and do not have enough number of the observation date of the far-field pattern.
}

This paper is organized as follows. In section 2, we adapt the Bayesian framework to inverse scattering problems with phaseless data. In section 3, we exhibit numerical results for recovering a disk\tcr{, a line crack and a kite-shaped obstacle}. Conclusions are given in section 4.

\section{Bayesian Framework}
In this paper we want to recover an unknown sound-soft obstacle from phaseless far-field patterns corresponding to a set of superposition of two plane waves. We propose the Bayesian approach to solve this inverse scattering problem. First of all, we set a suitable parameterization of the position and the shape of an obstacle. Then, in the Bayesian framework, we estimate the posterior distribution of the unknown obstacle parameters. By the Bayes' theorem \cite{Stuart_2015_Data_Assimilation, Stuart_2010_Bayesian_perspective}, the posterior distribution of these parameters can be obtained from the prior distribution and the likelihood function to be specified in this section. Numerically, we will adopt the Markov chain Monte Carlo method (MCMC) to get an approximation of the posterior distribution.

\subsection{Parameterization of the obstacle}

Since the boundary of the underlying obstacle is a $C^{2}$-smooth curve, we can represent or approximate its geometrical shape by a finite set $\mathbf{Z}$ of variables
\begin{equation}
    \mathbf{Z}:= ( z_{1}, z_{2}, \cdots, z_{N} )^{\top} \in \mathbb{R}^{N}, N\in \mathbb{N}_0. \label{parameter_obstacle}
\end{equation}
 For example, we can use four parameters $\mathbf{Z}:= ( z_{1}, z_{2}, z_{3}, z_{4} )^{\top}$ to represent a line segment, where $(z_{1}, z_{2})^{\top}$  and  $(z_{3}, z_{4})^{\top}$ denote respectively the two ending points, or we can use $\mathbf{Z}:= ( a_{1}$, $b_{1}$, $a_{2}$, $b_{2}$, $\cdots$, $a_{N}$, $b_{N} )^{\top}$ to approximate a star-shaped closed curve where $\{(a_j, b_j): j=1,\cdots, N\}$ stand for the Fourier coefficients in the
truncated Fourier expansion.

Recalling the incident waves $u^{in}_{\ell}(x), \ell = 1, 2, \cdots, L$ in the form of a set of superpositions of two plane waves \eqref{incident_wave}, we express the  (phased) far-field patterns  of the scattering model \eqref{Helmholtz_eq}-\eqref{SommerfeldCond}
 by
\begin{equation}
    u^{\infty}(\hat{\mathbf{x}}; \mathbf{Z}, \mathbf{d}_{0}, \mathbf{d}_{\ell}, k), \hskip .2 cm \ell = 1, 2, \cdots, L, \hspace{.2 cm} \hat{\mathbf{x}} \in \mathbb{S}. \label{far_field_pattern}
\end{equation}
Correspondingly, the phaseless far-field pattern are denoted by
\begin{equation}
    | u^{\infty}(\hat{\mathbf{x}}; \mathbf{Z}, \mathbf{d}_{0}, \mathbf{d}_{\ell}, k) |, \hskip .2 cm \ell = 1, 2, \cdots, L, \hspace{.2 cm} \hat{\mathbf{x}} \in \mathbb{S}, \label{phaseless_far_field_pattern}
\end{equation}
where $|\cdot|$ is the modulus of a complex number.

\subsection{Prior distribution}
By \eqref{parameter_obstacle}, the prior distribution of the obstacle parameters $\mathbf{Z}$ depends on the distribution of $z_{n}, n=1, 2, \cdots, N$. Let $\{z_{n}\}_{n=1}^N$ be independent variables with the prior density $\pi_{pr}^{n}$ and prior measure $\mu_{pr}^{n}$. Then the prior density $\pi_{pr}$ and prior measure $\mu_{pr}$ of $\mathbf{Z}$ are respectively given by
\begin{eqnarray}
    \pi_{pr}(\mathbf{Z}) &=& \prod_{n=1}^{N} \pi_{pr}^{n} ( z_{n} ), \label{prior_density} \\
    \mu_{pr}(d \mathbf{Z}) &=& \prod_{n=1}^{N} \mu_{pr}^{n} ( d z_{n} ). \label{prior_measure}
\end{eqnarray}
In this paper we assume that $z_{n}$ are random variables with the Gaussian distribution, that is, $$\mu_{pr}^{n} = \mathcal{N}( m_{n}, \sigma_{n} ),\quad n=1, 2, \cdots, N.$$ For simplicity, \tcr{we assume that $\sigma_{1} = \cdots = \sigma_{N} = \sigma_{pr}$, implying that $\mu_{pr} = \mathcal{N}( \mathbf{m}_{pr}, \sigma_{pr} \mathbf{I} )$, where $\mathbf{m}_{pr} = (m_{1}, m_{2}, \cdots, m_{N})^{\top}$ and} $\mathbf{I}\in \mathbb{R}^{N\times N}$ is the identity matrix.

\subsection{Observation of far-field pattern}
To bridge the parameterization \eqref{parameter_obstacle} of the obstacle and the associated phaseless far-field data \eqref{phaseless_far_field_pattern}, we define an operator $F: \mathbb{R}^{N}\rightarrow \tcr{C^{\infty}(\mathbb{S})}$ as
\begin{equation}
    | u^{\infty}(\hat{\mathbf{x}}) | = F(\mathbf{Z}), \hspace{.2 cm}  \hat{\mathbf{x}} \in \mathbb{S},   \label{forward_model}
\end{equation}
which can be regarded an abstract map from the space of obstacle parameters to the space of observation data in the continuous sense. From the well-posedness of forward scattering, $F$ is continuous but highly non-linear.

Let $G = (g_{1}, g_{2}, \cdots, g_{M})^{\top}: \tcr{C^{\infty}(\mathbb{S})}\rightarrow \mathbb{R}^M$ be a bounded linear observation operator with $g_{m}: \tcr{C^{\infty}(\mathbb{S})}\rightarrow \mathbb{R}_+$ given by
\[
    g_{m}( |u^{\infty}( \hat{\mathbf{x}} )| )=|u^{\infty}( \hat{\mathbf{x}}_{m} ) |,\quad m=1, 2, \cdots, M,
\]
where $\{\hat{\mathbf{x}}_{m}\}_{m=1}^M \subset \mathbb{S}$ is the set of discrete observation directions. Then the observation at the observation direction $\hat{\mathbf{x}}_{m}$ can be rephrased as
\begin{equation}
    y_{m} = g_{m}( | u^{\infty}( \hat{\mathbf{x}} ) | ) + \eta_{m} = | u^{\infty}( \hat{\mathbf{x}}_{m} ) | + \eta_{m}, \label{observation_with_noise}
\end{equation}
where $\eta_{m}$ represents the noise polluting the observation data at the direction $\hat{\mathbf{x}}_{m}$.

Set $\mathbf{Y} = (y_{1}, y_{2}, \cdots, y_{M})^{\top} \in \mathbb{R}^{M}$ and $\eta = (\eta_{1}, \eta_{2}, \cdots, \eta_{M})^{\top} \in \mathbb{R}^{M}$. Denote by $\mathcal{G} = G\circ F $ the map from the obstacle parameter space $\mathbb{R}^{N}$ to observation space $\mathbb{R}^{M}$, that is,
\begin{equation}
    \mathbf{Y} = \mathcal{G}( \mathbf{Z} ) + \eta, \qquad  Y, \eta \in \mathbb{R}^{M}, \hskip .1 cm  \mathbf{Z} \in \mathbb{R}^{N}. \label{forward_observation}
\end{equation}
Our inverse problem in this paper is to determine the obstacle parameters $\mathbf{Z}\in \mathbb{R}^{N}$ from the observation data $\mathbf{Y}\in \mathbb{R}^{M}$ with the noise pollution $\eta\in \mathbb{R}^{M}$.

\subsection{Likelihood}

We assume the observation pollution $\eta$ is independent of $u^{\infty}$ and drawn from the Gaussian distribution $\mathcal{N}( \mathbf{0}, \Sigma_{\eta} )$ with the density $\rho$, where $\Sigma_{\eta}\in \mathbb{R}^{M\times M}$ is a self-adjoint positive matrix. By the observation of the phaseless data with noise \eqref{forward_observation}, we can get the relationship $\mathbf{Y} | Z \sim \mathcal{N}( \mathcal{G}( \mathbf{Z} ), \Sigma_{\eta} )$. Define the model-data misfit function $\Phi( \mathbf{Z}; \mathbf{Y} ) : \mathbb{R}^{N} \times \mathbb{R}^{M} \rightarrow \mathbb{R} $ as
\begin{equation}
    \Phi( \mathbf{Z}; \mathbf{Y} ) = \frac{1}{2} | \mathbf{Y} - \mathcal{G}( \mathbf{Z} ) |^{2}_{\Sigma_{\eta}}, \label{model_data_misfit}
\end{equation}
where $| \cdot |_{\Sigma_{\eta}}  = | \Sigma_{\eta}^{-\frac{1}{2}} \cdot | $.  Hence, the likelihood function is given by
$$
    \rho\big( \mathbf{Y} - \mathcal{G}( \mathbf{Z} ) \big) = \frac{1}{ \big( (2\pi)^{M}\; \mbox{det}( \Sigma_{\eta} ) \big)^{1/2} }\; e^{- \Phi( \mathbf{Z}; \mathbf{Y} ) }.
$$
Furthermore,  the posterior density $\pi_{post}$ and the posterior measure $\mu_{post}$ are connected to  the prior measure $\mu_{pr}$ through the Radon-Nikodym derivative \cite{Durrent_2010_Probability}, given by
\begin{equation}
   \frac{d \mu_{post} }{ d \mu_{pr} } ( \mathbf{Z} ) \propto e^{- \Phi( \mathbf{Z}; \mathbf{Y} ) }. \label{posterior_measure}
\end{equation}

\subsection{Well-posedness of Bayesian framework}

The well-posedness arguments of \cite{BThanh_GOmar_2014_BayesianScattering, Stuart_2010_Bayesian_perspective} can be applied to deal with our inverse scattering problem with the Bayesian approach. In our phaseless case, we are required to justify the following Assumption \ref{Assumption_direct}, relying on regularity properties of the forward operator $\mathcal{G}$.
\begin{assumption}\label{Assumption_direct}  The map $\mathcal{G}:\mathbb{R}^{N} \rightarrow \mathbb{R}^{M}$ satisfies
    \begin{itemize}
      \item [(i)]  For every $\varepsilon>0$, there is an $\hat{M} = \hat{M}(\varepsilon) \in \mathbb{R}$ such that, for all $\mathbf{Z}\in \mathbb{R}^{N}$,
            \begin{equation*}
                | \mathcal{G}( \mathbf{Z} )  |_{\Sigma_{\eta}} \leq e^{ \varepsilon \| \mathbf{Z} \|^{2}_{2} + \hat{M} },
            \end{equation*}
          where $\| \cdot \|_{2}$ is the Euclidean norm.
      \item [(ii)] For every $r>0$, there is a $K = K(r) > 0$ such that, for all $\mathbf{Z}_{1}, \mathbf{Z}_{2} \in \mathbb{R}^{N}$ with \\ $ \max\big\{  \| \mathbf{Z}_{1} \|_{2},  \| \mathbf{Z}_{2} \|_{2}  \big\} < r$, it holds that
            \begin{equation*}
                | \mathcal{G}( \mathbf{Z}_{1} ) - \mathcal{G}( \mathbf{Z}_{2} )  |_{\Sigma_{\eta}} \leq K \| \mathbf{Z}_{1} - \mathbf{Z}_{2} \|_{2}.
            \end{equation*}
    \end{itemize}
\end{assumption}
We remark that there is no essential difference in proving Assumption \ref{Assumption_direct} (i) between the phased and phaseless inverse scattering problems. The Assumption \ref{Assumption_direct} (ii) follows directly from the triangle inequality $\big| |a|-|b| \big|\leq |a-b|$ for complex numbers $a, b\in \mathbb{C}$ and the corresponding assumption for phased inverse scattering problems. Hence, when $D$ is sound-soft \tcr{scatterer},   Assumption \ref{Assumption_direct} can be proved following the phased arguments of \cite{BThanh_GOmar_2014_BayesianScattering}; see also \cite{JGSun_2019_ScatterBayesian, YJWang_2015_BayesianScatteringInterior} for the proofs in the case of limited aperture data and for interior scattering problems. If $D$ is sound-soft crack, the same results can be verified by applying the Fr$\acute{\text{e}}$chet differentiability with respect to the boundary of the far field operator (\cite{RKress_1995_FDiffFarFieldOperate, Kress1995_ScatterOpenArc}).  The above assumptions together with the choice of the Gaussian prior measure (which satisfies $\mu_{pr}( \mathbb{R}^{N} ) = 1$) lead to well-posedness of the Bayesian inverse problem, which is a result of application of Lemma 2.8, Theorem 4.1, Theorem 4.2 and Theorem 6.31 in \cite{Stuart_2010_Bayesian_perspective}. Before stating the well-posedness (see Theorem \ref{Well_Posterior} below), we recall the Hellinger distance defined by

\begin{equation}
    d_{\text {Hell}}\left(\mu_{1}, \mu_{2}\right) :=\sqrt{\frac{1}{2} \int\left(\sqrt{\frac{d \mu_{1}}{d \mu_{0}}}-\sqrt{\frac{d \mu_{2}}{d \mu_{0}}}\right)^{2} d \mu_{0}},
    \label{Hellinger_distance}
\end{equation}
where $\mu_{1}, \mu_{2}$ are two measures that are absolutely continuous with respect to $\mu_{0}$.
%\end{definition}

\begin{theorem}\label{Well_Posterior}
    If the operator $\mathcal{G}$ satisfies the Assumption \ref{Assumption_direct} and the prior measure $\mu_{pr}$ satisfies $\mu_{pr}( \mathbb{R}^{N} ) = 1$, then the posterior measure $\mu_{post}$ is a well-defined probability measure on $\mathbb{R}^{N}$ and absolutely continuous with respect to prior measure $\mu_{pr}$. What's more, the posterior measure $\mu_{post}$ is Lipschitz in the data $\mathbf{Y} $, with respect to the Hellinger distance: if $\mu_{post}^{1}$ and $ \mu_{post}^{2} $ are two posterior measures corresponding to data $\mathbf{Y}_{1}$ and $\mathbf{Y}_{2}$, then there exists $C = C(r) > 0$ such that,
    \begin{equation*}
         d_{\text{Hell}}(\mu_{post}^{1}, \mu_{post}^{2})
         \leq C \|\mathbf{Y}_{1} - \mathbf{Y}_{2}\|_{2},
    \end{equation*}
    for all $\mathbf{Y}_{1}$, $\mathbf{Y}_{2}$ with $ \max\big\{  \| \mathbf{Y}_{1} \|_{2},  \| \mathbf{Y}_{2} \|_{2}  \big\} < r$.
\end{theorem}

\subsection{Preconditioned Crank-Nicolson (pCN) algorithm with random proposal variance}
This subsection is devoted to the numerical approximation of the posterior distribution. We adopt the Markov chain Monte Carlo method (MCMC) \cite{Brooks_2011_MCMC, Gamerman_2006_MCMC, Geyer_1992_MCMC} to generate a large number of samples subject to the posterior distribution. The numerical approximation of the posterior distribution of unknown obstacle parameters can be obtained by statistic analysis on these samples. The Metropolis-Hastings \cite{Hastings_1970, Metropolis_1953} algorithm will be used to construct MCMC samples. To improve the convergence rate of the MCMC method, we apply the preconditioned Crank-Nicolson algorithm \cite{Stuart_2013_pCN}.

According to the pCN algorithm, the new obstacle parameter $\mathbf{X}$  can be iteratively updated by the old parameter (initial guess) $\mathbf{Z}$ through the formula
\begin{equation}
    \mathbf{X} = \tcr{\mathbf{m}_{pr} + } ( 1-\beta^2 )^{1/2} \tcr{(\mathbf{Z} - \mathbf{m}_{pr}) } + \beta \omega, \label{pCN}
\end{equation}
where $\beta \in [0, 1]$ is the proposal variance coefficient and $\omega \sim  \mathcal{N}( \mathbf{0}, \Sigma_{pcn} ) $ is a zero-mean normal random vector with covariance matrix $\Sigma_{pcn}\in \mathbb{R}^{N\times N}$. We remark that it is important and very tricky to select a suitable $\beta$, because the value of $\beta$ dominates the proposal variance in the pCN algorithm. If $\beta\ll 1$ is small, the parameter $\mathbf{Z}$ will be updated slightly in the MCMC sequence, leading to a time-consuming iteration process to get the ergodic in the space of obstacle parameters. If $\beta$ is big, the parameter $\mathbf{Z}$ may stay at one state for quite a long time with a huge number of iterations in the MCMC method. Consequently, one cannot get enough number of samples to approximate the posterior distribution, due to the computational cost prohibition. To over come this difficulty, we recommend the pCN algorithm with a random proposal variance \cite{Stuart_2013_pCN} to obtain good MCMC sequences (see below for the description).

\begin{algorithm}\label{pCN_Random_Proposal} \textbf{pCN Algorithm with Random Proposal Variance}
\begin{itemize}
  \item Initialize $\mathbf{Z}_{0} \in \mathbb{R}^{N}$ and $\beta_{0} \in [0, 1]$.
  \item Repeat
        \begin{enumerate}
          \item Draw new obstacle parameter $\mathbf{X}$ from the old state $\mathbf{Z}_{j}$ by the pCN algorithm \eqref{pCN} with the proposal variance coefficient $\beta_{j}$ as:
               \begin{equation}
                   \mathbf{X} = \tcr{\mathbf{m}_{pr} + ( 1-\beta_{j}^2 )^{1/2} (\mathbf{Z}_{j} - \mathbf{m}_{pr} ) }+ \beta_{j} \omega, \hspace{.3cm} \omega \sim  \mathcal{N}( \mathbf{0}, \Sigma_{pcn} ); \label{pCN_RPV}
               \end{equation}
          \item Compute Hasting ratio $\alpha ( \cdot, \cdot ):\mathbb{R}^{N}\times \mathbb{R}^{N}\rightarrow[1,\infty) $ as:
                \begin{equation}
                    \alpha \big( \mathbf{X}, \mathbf{Z}_{j} \big) = min\{ 1, e^{ \Phi( \mathbf{Z}_{j}; \mathbf{Y} ) - \Phi( \mathbf{X}; \mathbf{Y} )  } \}; \label{alpha_Z}
                \end{equation}
          \item Accept or reject $\mathbf{X}$: draw $U \sim \mathcal{U}(0,1)$ and then update $\mathbf{Z}_{j}$ by the criterion
                \begin{equation}
                    \mathbf{Z}_{j+1} =
                                \left\{
                                        \begin{array}{ll}
                                                \mathbf{X}, & \hbox{if\; $U \leq \alpha \big( \mathbf{X}, \mathbf{Z}_{j} \big) $,} \\
                                                \mathbf{Z}_{j}, & \hbox{if\, otherwise;}
                                        \end{array}
                                \right. \label{update_Z}
                \end{equation}
          \item Generate new proposal variance coefficient $\beta_{j+1}$ from $\beta_{j}$. First we set
                \begin{equation}
                   \beta_{new} = ( 1-\gamma^2 )^{1/2} \beta_{j} + \gamma ( \omega_{\beta} - 0.5 ), \hspace{.3cm} \omega_{\beta} \sim \mathcal{U}(0,1), \label{pCN_beta}
                \end{equation}
                with $\gamma \in [0, 1]$. In our case we choose $\gamma = 0.1$. Then $\beta_{j+1}$ can be updated by
                \begin{equation}
                    \beta_{j+1} =
                                \left\{
                                        \begin{array}{ll}
                                                \beta_{new}, & \mbox{if}\quad \beta_{new} \in [0, 1], \\
                                                - \beta_{new}, & \mbox{if}\quad \beta_{new} < 0, \\
                                                \beta_{new} - 1, & \mbox{if}\quad\beta_{new} > 1 .
                                        \end{array}
                                \right. \label{update_beta}
                \end{equation}
        \end{enumerate}

  \item \tcr{Select $\mathbf{Z}_{\hat{j}}, \  \hat{j} = J_{1} + (\tilde{j} - 1 ) J_{2}, \   \tilde{j} = 1, 2, \cdots, J_{3}$, $J_{1}, J_{2}, J_{3} \in \mathbb{N}_{0}$.}
\end{itemize}
\end{algorithm}

\tcr{In the Algorithm \ref{pCN_Random_Proposal}, the randomness of} the proposal variance has the potential advantage of including the possibility of large and small proposal variance coefficients $\beta$. The large proposal variance coefficient $\beta$ helps the pCN algorithm explore the state space efficiently, and the small proposal variance coefficient $\beta$ protects the iterations from dropping into a fixed state. Then the random proposal variance gives rise to ergodic Markov chains.

\tcr{The sequence $\mathbf{Z}_{\hat{j}}$ in the Algorithm \ref{pCN_Random_Proposal} is selected to approximate the posterior distribution. Here, $J_{1}$ is the number of initial states;
we take every $J_{2}$ sates to guarantee the selected sates are independent;  $J_{3}$ is the number of the total selected states.}

\section{Numerical Examples}
In this section we exhibit numerical examples to demonstrate the effectiveness of the method described in the previous section. To save computational costs, we consider \tcr{three types of acoustically sound-soft scatterers} in two dimensions:
\begin{itemize}
    \item Sound-soft disks with unknown centers and radii;
    \item Line cracks with unknown starting and ending points;
    \item \tcr{Kite-shaped obstacles with unknown position and shape.}
\end{itemize}
There are totally three unknown parameters for disks, four parameters for cracks \tcr{and six parameters for kite-shaped obstacle} in 2D, implying that \tcr{the unknown parameters always} lie in a finite space with low dimensions.

With the definition of the observation \eqref{forward_observation}, we construct two types of observations. \tcr{Since the observation noise $\eta$ in \eqref{forward_observation} is assumed to be a Gaussian random vector, the zero vector is a special sample of the observation noise. Then in the first type observations, we consider an ideal model where the phaseless far-field data corresponding to the exact obstacle are polluted by this special sample ($\eta = \mathbf{0}$) of observation noise. In the second type observations, we consider a practical model with noise-polluted far-field data, which will be used to discuss the robustness of the numerical method in practical applications. The Hausdorf difference between the reconstructed and  exact scatterers indicates the accuracy of our numerical method.}

\subsection{Disk}
In this subsection, we assume the underlying obstacle is a sound-soft disk. The inverse problem of recovering disks arises from, for instance, the polarization model of time-harmonic electromagnetic scattering from perfectly conducting cylinders whose cross-section is a disk. The parameterization of a disk is given by
\begin{equation}
    \mathbf{Z}:= ( z_{1}, z_{2}, z_{3} )^{\top} = ( x_{1}, x_{2}, \tcr{\log{r}} )^{\top}, \label{parameter_disk}
\end{equation}
where $( x_{1}, x_{2} )^\top $ is the center and $r$ is the radius of the disk. Since $r>0$, we assume that $r$ is a lognormal random variable\tcr{, i.e., $z_{3} = \log{r}$ is a Gaussian random variable}.

Let the incident wave be given by the sum of two plane waves of the form \eqref{incident_wave}. If the disk is located at the origin, it is well-known that the corresponding far-field pattern incited by the plane wave $u^{in}(x)=e^{ikx\cdot \mathbf{d}_{\ell}}$ is given by the convergent series
\begin{equation}
    u^{\infty}(\hat{\mathbf{x}}; \mathbf{Z}, \mathbf{d}_{\ell}, k)
    =  -e^{-i\frac{\pi}{4}}\sqrt{\frac{2}{\pi k}} \left[ \frac{J_{0}(kr)}{H_{0}^{(1)}(kr)} + 2\sum_{n=1}^{\infty}\frac{J_{n}(kr)}{H_{n}^{(1)}(kr)}\cos(n\theta_{\ell}) \right],\quad \ell=0,1,\cdots,L.
\end{equation}
Here, $\theta_{\ell}=\angle(\hat{\mathbf{x}}, \mathbf{d}_{\ell})$ denotes the angle between the observation direction $\hat{\mathbf{x}}$ and the incident direction $\mathbf{d}_{\ell}$, $J_{n}(\cdot)$ is the Bessel function of order $n$ and $H_{n}^{(1)}(\cdot)$ is the Hankel function of the first kind of order $n$. If the disk is located at $(x_1, x_2)^\top \in \mathbb{R}^2$, by the translational formula \eqref{translation} and the linear superposition principle, the exact far-field pattern with phase information of the scattered waves can be expressed as
\begin{equation}
\begin{split}
     u^{\infty}(\hat{\mathbf{x}}; \mathbf{Z}, \mathbf{d}_{0}, \mathbf{d}_{\ell}, k)
     = & -e^{-i\frac{\pi}{4}}\sqrt{\frac{2}{\pi k}} \left[ \frac{J_{0}(kr)}{H_{0}^{(1)}(kr)} + 2\sum_{n=1}^{\infty}\frac{J_{n}(kr)}{H_{n}^{(1)}(kr)}\cos(n\theta_{0}) \right] e^{\mathrm{i} k ( x_{1}, x_{2} )^{\top}\cdot ( \mathbf{d}_{0} - \hat{\mathbf{x}}) } \\
     & -e^{-i\frac{\pi}{4}}\sqrt{\frac{2}{\pi k}} \left[ \frac{J_{0}(kr)}{H_{0}^{(1)}(kr)} + 2\sum_{n=1}^{\infty}\frac{J_{n}(kr)}{H_{n}^{(1)}(kr)}\cos(n\theta_{\ell}) \right] e^{\mathrm{i} k ( x_{1}, x_{2} )^{\top}\cdot ( \mathbf{d}_{\ell} - \hat{\mathbf{x}}) }, \\ & \text{\ \ \ \ \ \ \ \ \ \ \ \ } \ell = 1, 2, \cdots, L, \hspace{.2 cm} \hat{\mathbf{x}} \in \mathbb{S}. \label{infinity_scattering_wave_disk}
\end{split}
\end{equation}
Note that the first line on the right hand side of (\ref{infinity_scattering_wave_disk}) denotes the far-field pattern corresponding to the incoming plane wave $e^{ikx\cdot \mathbf{d}_{0}}$, while the second line corresponding to $e^{ikx\cdot \mathbf{d}_{\ell}}$.

The following theorem states that our phaseless data set is sufficient to uniquely identify a sound-soft disk.
\begin{theorem}\label{TH}
    Let $k>0$ and $\mathbf{d}_{0}\in \mathbb{S}$ be fixed. Then the data $\{|u^{\infty}(\hat{\mathbf{x}}; \mathbf{Z}, \mathbf{d}_{0}, \mathbf{d}, k)|: \mathbf{d}\in \mathbb{S}\}$ uniquely determine a sound-soft disk (that is, the center and  radius of a disk).
\end{theorem}
\begin{proof}
Suppose that $D_j:=\{z=(z_1,z_2)^\top \in \mathbb{R}^2: \| z-x^{(j)} \|_{2}  < r_j\}$ ($j=1,2$) are two sound-soft disks centered at $x^{(j)}=(x^{(j)}_1,x^{(j)}_2)^\top \in \mathbb{R}^2$ with the radius $r_j>0$ ($j=1,2$). Set $\mathbf{Z}^{(j)}=(x^{(j)}_1, x^{(j)}_2, r_j)^\top$ and denote by $u^{\infty}(\hat{\mathbf{x}}; \mathbf{Z}^{(j)}, \mathbf{d}_{0}, \mathbf{d}, k)$ the far-field data corresponding to $D_j$ and the incident wave (\ref{incident_wave}). Suppose that the phaseless far-field pattern are identical, i.e.,
\begin{equation}\label{eq1}
    |u^{\infty}(\hat{\mathbf{x}}; \mathbf{Z}^{(1)}, \mathbf{d}_{0}, \mathbf{d}, k)|=|u^{\infty}(\hat{\mathbf{x}}; \mathbf{Z}^{(2)}, \mathbf{d}_{0}, \mathbf{d}, k)| \qquad\mbox{for all}\quad
    \mathbf{d}, \hat{\mathbf{x}}\in \mathbb{S}.
\end{equation}
In particular, choosing $\mathbf{d}=\mathbf{d}_0$ in the previous relation yields
\[
    |u^{\infty}(\hat{\mathbf{x}}; \mathbf{Z}^{(1)}, \mathbf{d}_{0}, \mathbf{d}_{0}, k)|=|u^{\infty}(\hat{\mathbf{x}}; \mathbf{Z}^{(2)}, \mathbf{d}_{0}, \mathbf{d}_0, k)|=2|u^{\infty}(\hat{\mathbf{x}}; \mathbf{Z}^{(j)}, \mathbf{d}_{0}, k)|,\quad j=1,2,
\]
for all $\hat{\mathbf{x}}\in \mathbb{S}$, where $u^{\infty}(\hat{\mathbf{x}}; \mathbf{Z}^{(j)}, \mathbf{d}_{0}, k) $ stands for the far-field pattern corresponding to the plane wave $e^{ikx\cdot \mathbf{d}_{0}}$ incident onto $D_j$. This implies that
\[
    |u^{\infty}(\hat{\mathbf{x}}; \mathbf{Z}^{(1)}, \mathbf{d}_{0}, k)|=|u^{\infty}(\hat{\mathbf{x}}; \mathbf{Z}^{(2)}, \mathbf{d}_0, k)|
    \qquad\mbox{for all}\quad
    \hat{\mathbf{x}}\in \mathbb{S}.
\]
Next, we shift the center of the disk $D_j$ to the origin and set $\mathbf{Z}^{(j)}_0:=(0,0, r_j)^\top$. Recalling the translational formula \eqref{translation}, we obtain
\[
    |u^{\infty}(\hat{\mathbf{x}}; \mathbf{Z}_0^{(1)}, \mathbf{d}_{0}, k)|=|u^{\infty}(\hat{\mathbf{x}}; \mathbf{Z}_0^{(2)}, \mathbf{d}_0, k)|
    \qquad\mbox{for all}\quad
    \hat{\mathbf{x}}\in \mathbb{S}.
\]
Since the shifted disks with the parameters $\mathbf{Z}_0^{(j)}$ are rotationally invariant, the far-field pattern $u^{\infty}(\hat{\mathbf{x}}; \mathbf{Z}_0^{(j)}, \mathbf{d}, k)$ only depends on the angle between the incident direction $\mathbf{d}$ and the observation direction $\hat{\mathbf{x}}$. For any $\mathbf{d}\in\mathbb{S}$, there exist a orthogonal matrix $Q$ such that $\mathbf{d}=Q\mathbf{d}_0$. It then follows that
(see e.g., \cite[Chapter 5.1]{DColton_2013_InverseAcousticScattering})
\[
    u^{\infty}(\hat{\mathbf{x}}; \mathbf{Z}_0^{(j)}, \mathbf{d}, k)=
    u^{\infty}(\hat{\mathbf{x}}; \mathbf{Z}_0^{(j)}, Q\mathbf{d}_0, k)=
    u^{\infty}(Q\hat{\mathbf{x}}; \mathbf{Z}_0^{(j)}, \mathbf{d}_0, k), \quad\forall \;  \hat{\mathbf{x}}\in\mathbb{S}.
\]
Combining the previous two identities yields
\[
    |u^{\infty}(\hat{\mathbf{x}}; \mathbf{Z}_0^{(1)}, \mathbf{d}, k)|=|u^{\infty}(\hat{\mathbf{x}}; \mathbf{Z}_0^{(2)}, \mathbf{d}, k)|
    \qquad\mbox{for all}\quad
    \hat{\mathbf{x}}, \mathbf{d}\in \mathbb{S},
\]
which together with the translational formula implies
\begin{equation}\label{eq2}
    |u^{\infty}(\hat{\mathbf{x}}; \mathbf{Z}^{(1)}, \mathbf{d}, k)|=|u^{\infty}(\hat{\mathbf{x}}; \mathbf{Z}^{(2)}, \mathbf{d}, k)|
    \qquad\mbox{for all}\quad
    \hat{\mathbf{x}}, \mathbf{d}\in \mathbb{S}.
\end{equation}
As a consequence of \cite[Theorem 2.2]{BZhang_2018_UniquenessScatterPhaseless}, the relations (\ref{eq1}) and (\ref{eq2}) lead to the coincidence of $D_1$ and $D_2$, which proves Theorem \ref{TH}.
\end{proof}

To apply the pCN algorithm \eqref{pCN} or \eqref{pCN_RPV}, we need to set the key parameters $\beta$ and $\Sigma_{pcn}$. For sound-soft disks, we set $\beta_{j} = \beta = 0.1$ and let $\Sigma_{pcn} = \mathbf{I}$ be the $N$-by-$N$ identity matrix. Then the proposal is given by
\begin{equation}
    \mathbf{X} = \tcr{\mathbf{m}_{pr} + } \sqrt{0.99} \tcr{(\mathbf{Z}_{j} - \mathbf{m}_{pr})} + 0.1 \omega, \hspace{.3cm} \omega \sim  \mathcal{N}( \mathbf{0}, \mathbf{I} ). \label{pCN_disk}
\end{equation}

We describe the settings of our computational performance as follows:
\begin{itemize}
    \item Unless otherwise specified, the wave number is always taken as $k=1$;

    \item The incident directions are
        \begin{equation}
            \tcr{
            \mathbf{d}_{\ell}
            =(\cos\theta_{\ell}, \sin\theta_{\ell}),
            \quad
            \theta_{\ell} = -\frac{\pi}{2} + \frac{2\pi \ell}{L+1},
            \quad
            \ell=0,1,\cdots,L;
            \label{incident_disk}
            }
        \end{equation}

    \item The observation directions are
        \begin{equation}
            \tcr{
            \hat{\mathbf{x}}_{m}
            = \big(\cos\theta_{m}, \sin\theta_{m} \big),
            \quad
            \theta_{m} = -\frac{\pi}{2} + \frac{2\pi m}{M},
            \quad
            m=1,2,\cdots,M;
            \label{observation_disk}
            }
        \end{equation}

    \item To compute the far-field pattern \eqref{infinity_scattering_wave_disk}, we truncate the infinite series of (\ref{infinity_scattering_wave_disk}) by using the Bessel and first-kind Hankel functions of order $n=0,1,2,\cdots, 100$;

    \item \tcr{In the Algorithm \ref{pCN_Random_Proposal}, we choose $J_{1}=9000$, $J_{2}=5$, $J_{3}=201$;}

    \item In the setting of the prior distribution $\pi_{pr}$, we assume $\sigma_{pr} = 1$;

    \item For the observation $\mathbf{Y}_{\ell}$ corresponding to the incident wave $u^{in}_{\ell}(x)$ in \eqref{incident_wave}, $\ell = 1, 2, \cdots, L $, we assume the observation pollution $\eta^{\ell} = ( \eta^{\ell}_{1}, \eta^{\ell}_{2}, \cdots,  \eta^{\ell}_{M} )^{\top}$ is a M-dimensional Gaussian variable, given by
      \begin{equation}
            \eta^{\ell}_{m} = \sigma_{\eta} \times \left| u^{\infty}(\hat{\mathbf{x}}_m; \mathbf{Z}, \mathbf{d}_{0}, \mathbf{d}_{\ell}, k) \right| \,\omega^{\ell}_{m}, \label{observation_noise}
      \end{equation}
    where $\omega^{\ell}_{m} \sim \mathcal{N}(0, 1)$, $m = 1, 2, \cdots, M$ and $\sigma_{\eta}$ is the noise coefficient. In our numerical tests we choose $ \sigma_{\eta} = 3\%, 6\%, 9\% $. In other words, we take $\eta^{\ell} \sim \mathcal{N}( \mathbf{0}, \mathbf{\Sigma_{\eta^{\ell}}} )$ and the diagonal matrix $\mathbf{\Sigma_{\eta^{\ell}}} = \mbox{diag}( \sigma^{\ell}_{1},  \sigma^{\ell}_{2}, \cdots, \sigma^{\ell}_{M} )$ with $\sigma^{\ell}_{m} = \big( \sigma_{\eta}\times \left| u^{\infty}(\hat{\mathbf{x}}_m; \mathbf{Z}, \mathbf{d}_{0}, \mathbf{d}_{\ell}, k)\right| \big)^{2}, m = 1, 2, \cdots, M$;

    \item Unless otherwise specified, the accurate obstacle parameters are set as $\hat{\mathbf{Z}} = (\hat{x}_{1}, \hat{x}_{2}, \hat{r})^{\top} = (1, 0.25, 0.12)^{\top}$, that is, a disk centered at $(1, 0.25)^{\top}$ with the radius $0.12$;

    \item \tcr{We assume the initial guess is a disk centered at origin $(0, 0)^{\top}$ with the radius $0.05$. Then the mean of the prior distribution is $\mathbf{m}_{pr} = (0, 0, \log{0.05})^{\top}$.}

\end{itemize}

In the first part, we adopt the ideal setting. Noting that in the formula \eqref{observation_noise} $\omega^{\ell}_{m} \sim \mathcal{N}(0, 1)$, we can gain a special sample of the observation noise with $\omega^{\ell}_{m} = 0$ and $\sigma_{\eta} = 3\%$\tcr{, $\ell = 1, 2, \cdots, L$, $m = 1, 2, \cdots, M$}. This would help us to investigate the accuracy of the numerical method and to verify the above uniqueness result in Theorem \ref{TH}.

At first, we discuss the accuracy of the numerical solutions for different choice of $L$ and $M$. Recall that the parameter $L$ denotes the number of incident waves and $M$ the number of observation directions. In Table \ref{z_r_L_M_table}, we show the mean, standard deviation and the relative error of the reconstructed parameters with different choice of $L, M$.  The numerical solutions of the recovered centers and radii  are shown in Figures \ref{z_L_M} and \ref{r_L_M}. The histograms of numerical solutions of the  centers and radii are shown in Figures \ref{hist_x_L_M},  \ref{hist_y_L_M} and \ref{hist_r_L_M}, respectively.

Based on these results, we find that the reconstructed parameters are getting more accurate as the number of incident or observation directions become larger. The numerical solutions with  $(L, M) = (32, 32)$, $(32, 16)$, $(16, 32)$, $(16, 64)$, $(8, 64)$ are relatively inaccurate, since the resulted relative errors are \tcr{larger} than $5\%$ in these cases. In contrast, the relative error is less $1\%$ if we choose $L$ and $M$ large enough such as  $(L, M) = (64, 128)$, $(64, 64)$, $(32, 128)$. On the other hand, it can be observed from Table \ref{z_r_L_M_table} that the standard deviation decreases as $L$ or $M$ increases. Table \ref{z_r_L_M_table_worse} and  Figure \ref{z_r_vs_L_M_worse} illustrate that small $L$ and $M$ may lead to unreliable reconstructions.

\begin{table}[htbp]
    \centering
    \caption{Mean, standard deviation and relative error of $(x_{1}, x_{2}), r$ vs $L$ and $M$.}
    \label{z_r_L_M_table}
    \begin{tabular}{cc|c|c|c}
    \hline
    \  $L$ & $M$ & mean & standard deviation & relative error \\
    \hline
     32 & 128 & (0.9976, 0.2494), 0.1199 & (0.0402, 0.0173), 0.0026 & \tcr{$(0.24\%, 0.24\%), \ 0.07\%$} \\
    \hline
    32 & 32 & (0.9192, 0.2804), 0.1203 & (0.1893, 0.0847), 0.0100 & \tcr{$(8.08\%, 12.15\%), \ 0.28\%$} \\
    \hline
    32 & 16 & (0.9222, 0.2833), 0.1209 & (0.1997, 0.0808), 0.0118 & \tcr{$(7.78\%, 13.30\%), \ 0.77\%$} \\
    \hline
    64 & 128 & (0.9941, 0.2514), 0.1199 & (0.0278, 0.0144), 0.0016 & \tcr{$(0.59\%, 0.56\%), \ 0.10\%$} \\
    \hline
    32 & 64 & (0.9934, 0.2542), 0.1200 & (0.0606, 0.0287), 0.0034 & \tcr{$(0.66\%, 1.67\%), \ 0.03\%$} \\
    \hline
     16 & 32 & (0.9023, 0.2794), 0.1201 & (0.2207, 0.0841), 0.0112 & \tcr{$(9.77\%, 11.74\%), \ 0.11\%$} \\
    \hline
     64 & 64 & (0.9934, 0.2521), 0.1198 & (0.0399, 0.0204), 0.0025 & \tcr{$(0.66\%, 0.83\%), \ 0.18\%$} \\
    \hline
    16 & 64 & (0.9290, 0.2660), 0.1196 & (0.1857, 0.0763), 0.0104 & \tcr{$(7.10\%, 6.41\%), \ 0.36\%$} \\
    \hline
    8 & 64 & (0.9127, 0.2661), 0.1187 & (0.1830, 0.0828), 0.0115 & \tcr{$(8.73\%, 6.43\%), \ 1.07\%$} \\
    \hline
    \end{tabular}
\end{table}

\begin{table}[htbp]
    \centering
    \caption{Mean and relative error of $(x_{1}, x_{2}), r$ vs small $L$ and $M$.}
    \label{z_r_L_M_table_worse}
    \begin{tabular}{cc|c|c}
    \hline
    \  $L$ & $M$ & mean & relative error \\
    \hline
    4 & 2 & (3.0494, 0.9982), 0.2643  & \tcr{(2.05\%, 2.99\%), \ 1.20\%} \\
    \hline
    4 & 8 & (3.0457, 0.9949), 0.2717 &  \tcr{(2.05\%, 2.98\%), \ 1.26\%} \\
    \hline
    8 & 4 & (2.9069, 1.1143), 0.2732  & \tcr{(1.91\%, 3.46\%), \ 1.28\%} \\
    \hline
    \end{tabular}
\end{table}

\begin{figure}[htbp]
  \centering
  \includegraphics[width = 4in]{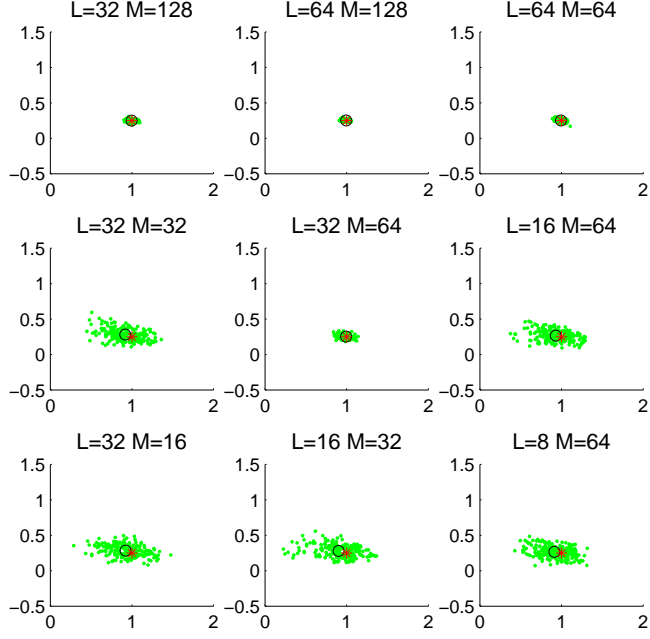}
  \caption{Reconstructions of the center of a sound-soft disk with different $L$ and $M$ at the idea setting. The red star $\textcolor{red}{\ast}$ denotes the accurate center, the green dots $\textcolor{green}{\cdot}$ are the numerical centers and the black $\circ$ is the mean of the numerical centers.}
  \label{z_L_M}
\end{figure}

\begin{figure}[htbp]
  \centering
  \includegraphics[width = 4in]{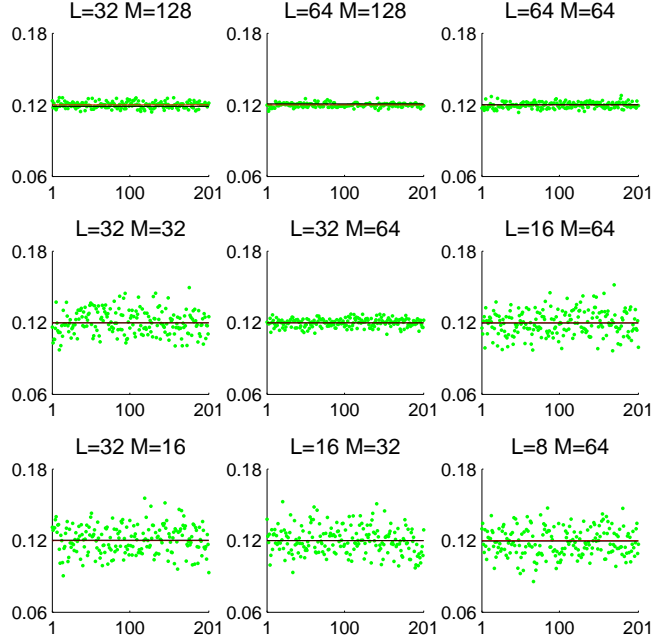}
  \caption{Reconstructions of the radius of a sound-soft disk with different $L$ and $M$ at the idea setting. The red line is the accurate radius, the black line is the mean of the numerical radii and the green dots $\textcolor{green}{\cdot}$ are the numerical radii.}
  \label{r_L_M}
\end{figure}

\begin{figure}[htbp]
  \centering
  \includegraphics[width = 4in]{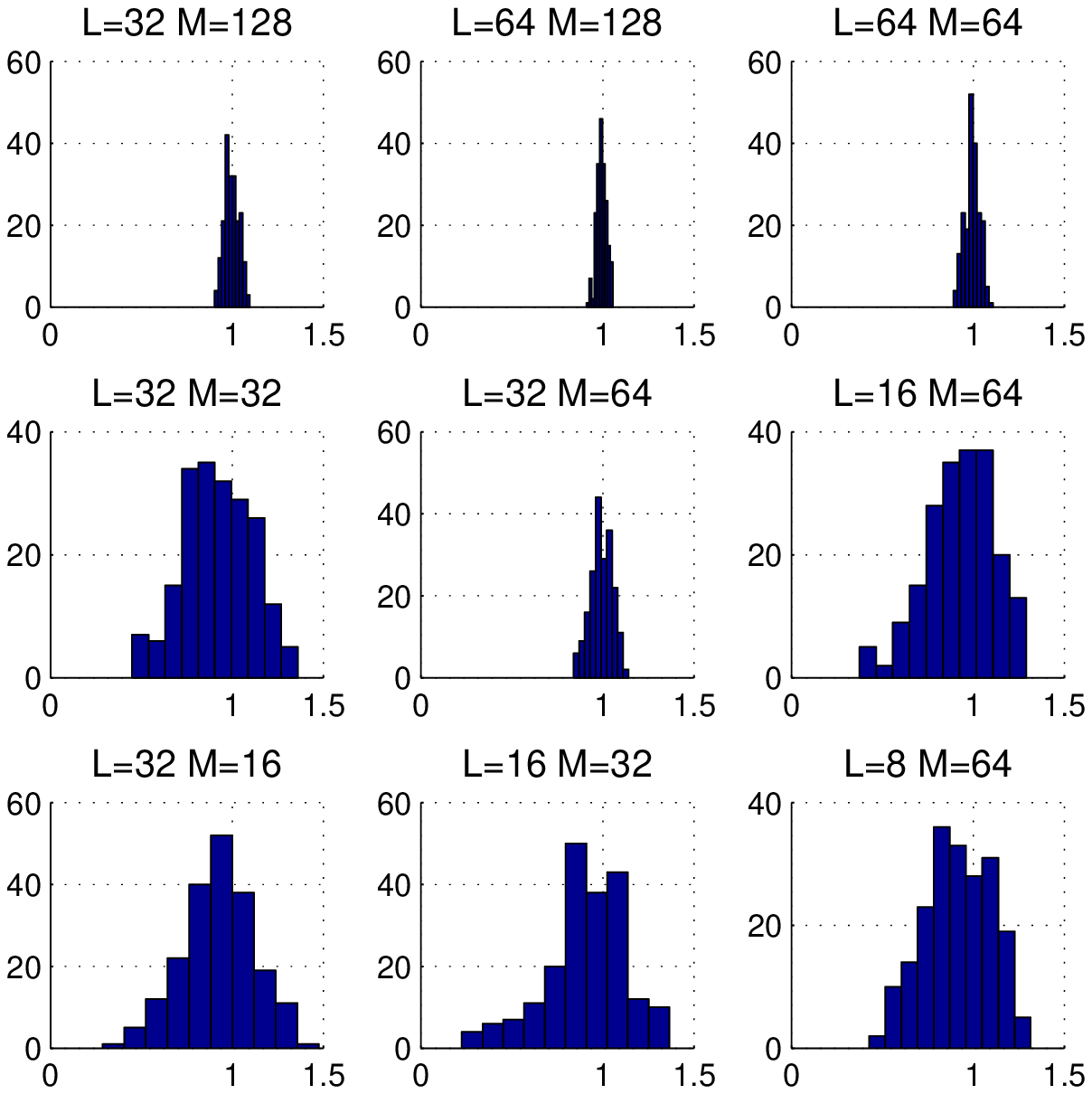}
  \caption{ Histograms of 201 reconstruction results of the $x_1$-component of the center with different $L$ and $M$ at the idea setting.}
  \label{hist_y_L_M}
\end{figure}

\begin{figure}[htbp]
  \centering
  \includegraphics[width = 4in]{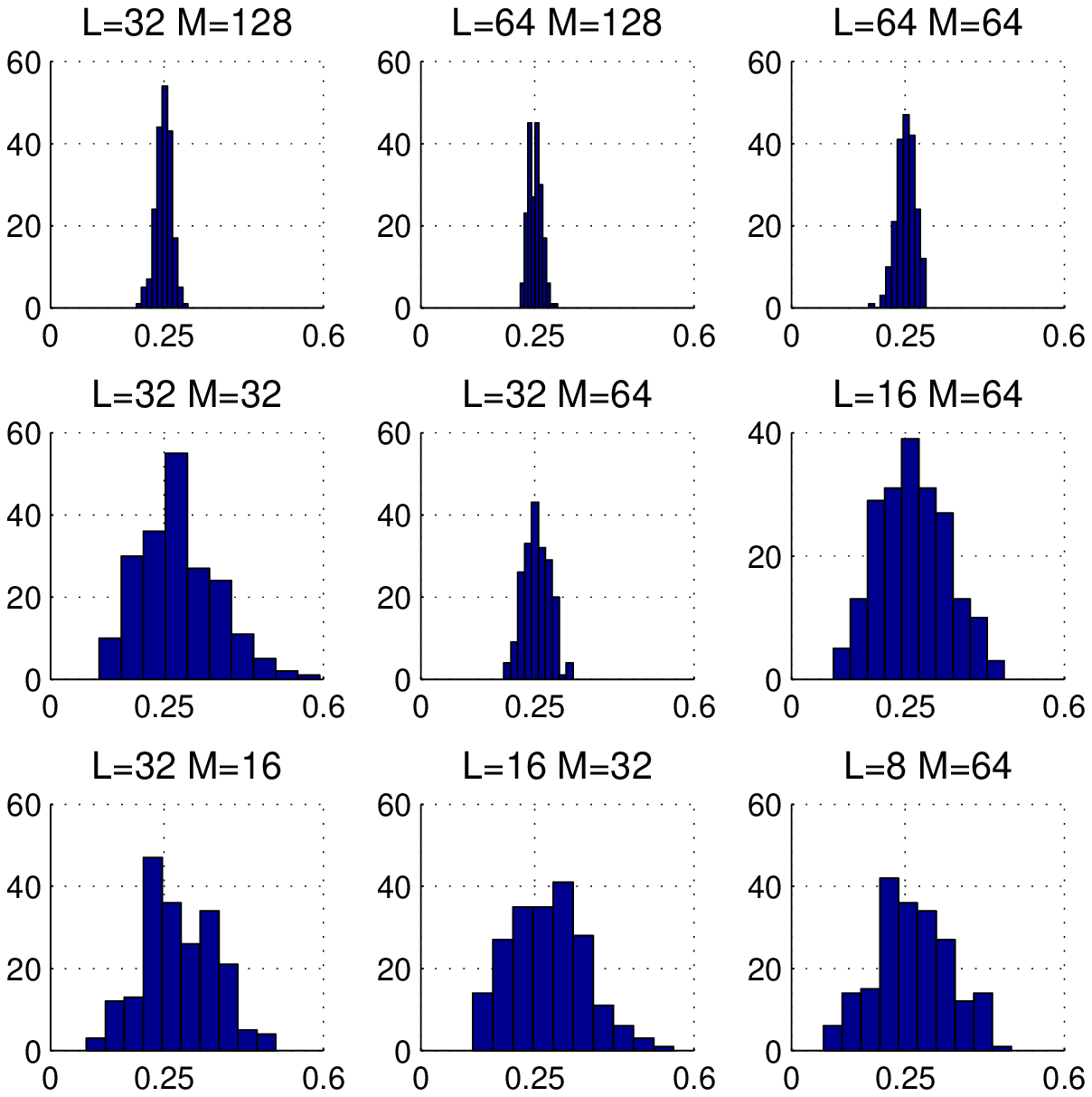}
  \caption{Histograms of 201 reconstruction results of the $x_2$-component of the center with different $L$ and $M$ at the idea setting.}
  \label{hist_x_L_M}
\end{figure}

\begin{figure}[htbp]
  \centering
  \includegraphics[width = 4in]{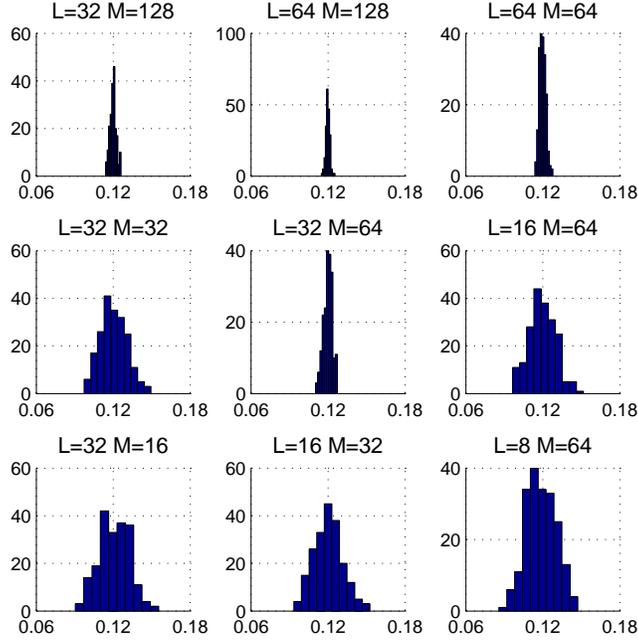}
  \caption{Histograms of 201 reconstruction results of the radius of a sound-soft disk with different $L$ and $M$ at the idea setting.}
  \label{hist_r_L_M}
\end{figure}

\begin{figure}[htbp]
  \centering
  \includegraphics[width = 5in]{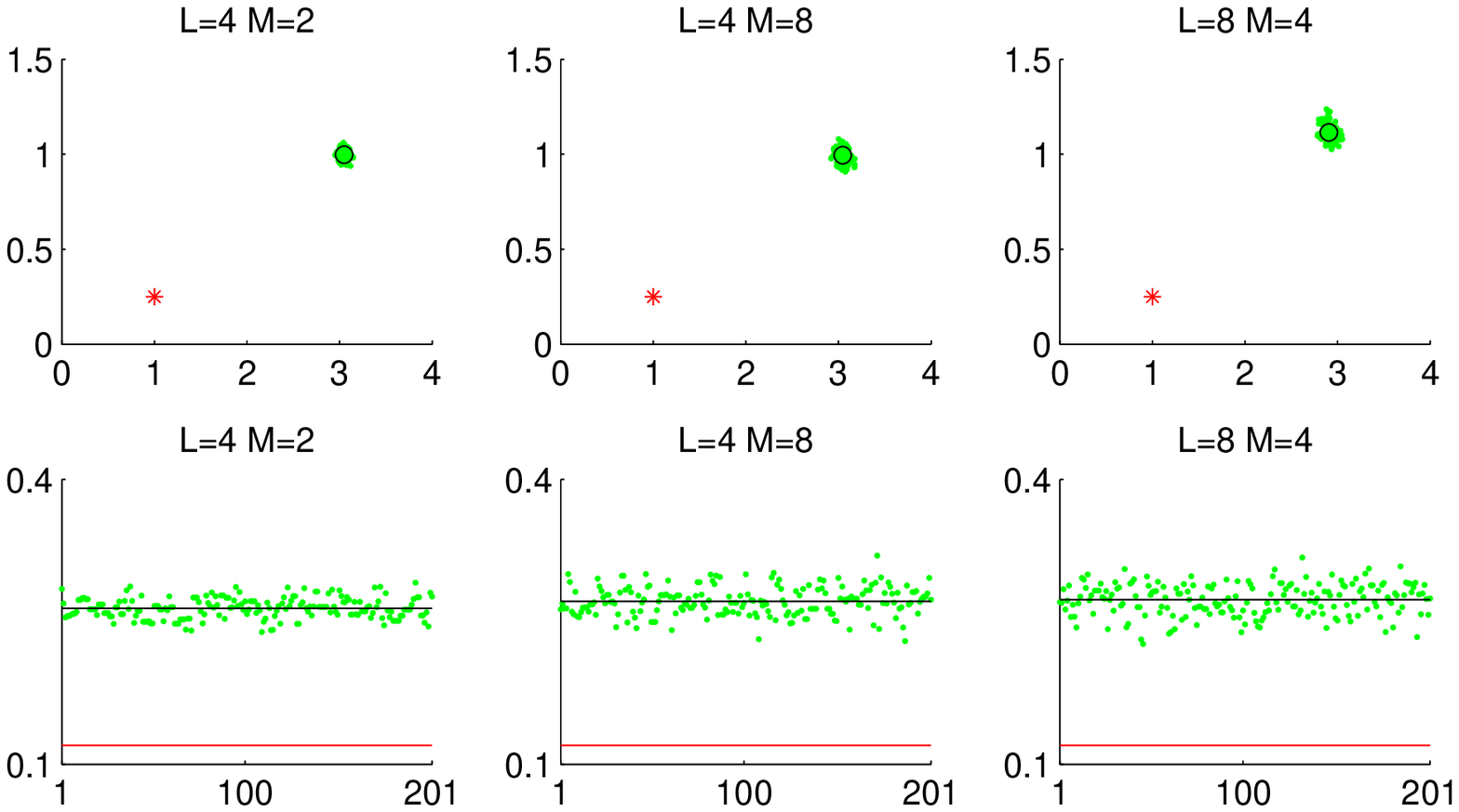}
  \caption{ Reconstructions of the center (top) and radius (bottom) of a sound-soft disk with a small number of incident waves ($L$) and observation directions ($M$) at the idea setting. The red stars $\textcolor{red}{\ast}$ (resp. lines) are the accurate center (resp. radius); the green dots $\textcolor{green}{\cdot}$ are the numerical reconstructions;  the black dots $\circ$ (lines) are the mean values of the center (resp. radius).}
  \label{z_r_vs_L_M_worse}
\end{figure}

In the following we suppose that the location of the center $(\hat{x}_{1}, \hat{x}_{2})^\top = (2, 2)^\top$ is known and the knowledge of the radius needs to be recovered. Since only one parameter of the scatter remains unknown, we make use of minimal number of incident and observation directions by setting $L=M=1$. In our tests we set incident directions $\mathbf{d}_{0} = (0, -1)^\top, \mathbf{d}_{1} = (0, 1)^\top$, observation direction $\hat{\mathbf{x}}_{1} = (0, -1)^\top$ and accurate radius $\hat{r}=1$. The numerical approximations of radius $r$ vs different wave numbers $k$ are exhibited in Figure \ref{r_vs_k_all}. For each fixed $k$, we plot the phaseless far-field pattern $|u^{\infty}(\hat{\mathbf{x}}_{1}) |$ against the radius $r$ in Figure \ref{u_vs_r_fixed_k}.

From the numerical results we conclude that an accurate approximation of the radius can be obtained if the wave number $k$ is less than a threshold. It is seen from Figure \ref{u_vs_r_fixed_k} that the function $r\rightarrow |u^\infty(\hat{\mathbf{x}}_1)|$ is monotonically increasing in $(0, R(k))$ where  $R(k)\rightarrow 0^+$ as $k\rightarrow +\infty$. This suggests that for large $k$ such as $k=97, 200, 2000$, there are more than one radii corresponding to the measured phaseless far-field pattern at $\hat{\mathbf{x}}_1$. Hence, the reconstructed radii are inaccurate. These findings are consistent with the uniqueness result of \cite{BZhang_2010_UniqueBallSingleFar}, which states that a sound-soft disk can be uniquely determined from the phaseless far-field pattern at one observation direction, provided the radius is sufficiently small for a fixed wave number. The monotonicity property of the backscattered phaseless data with respect to the radius was rigorously justified in \cite{BZhang_2010_UniqueBallSingleFar}.

\begin{figure}[htbp]
  \centering
  \includegraphics[width = 6in]{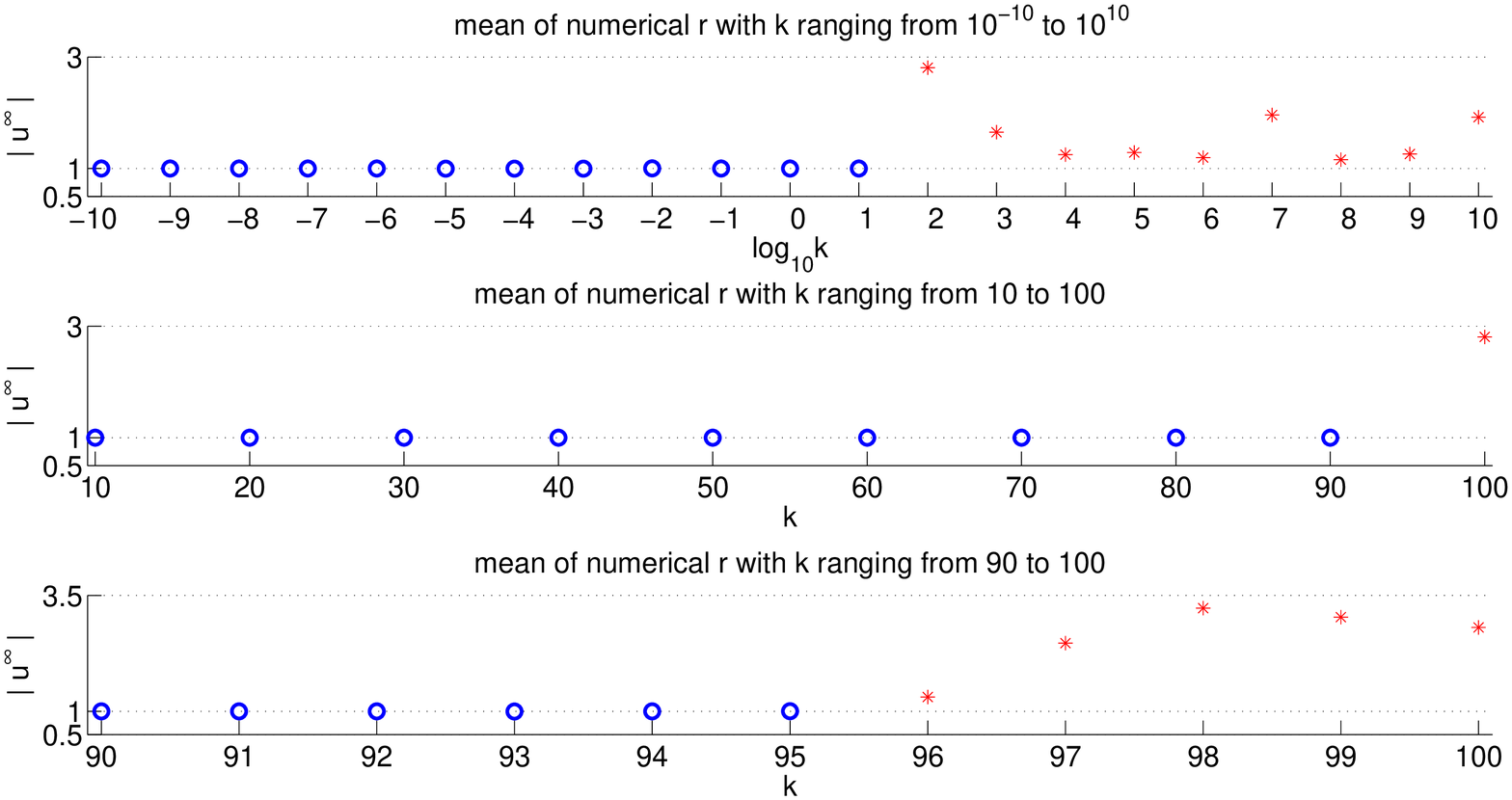}
  \caption{\tcr{The reconstructed radii with different wavenumbers $k$ ranging from $10^{-10}$ to $10^{10}$ (top), from 10 to 100 (middle) and those from 90 to 100 (bottom). Both the blue dots $\textcolor{blue}{\circ}$ and red starts $\textcolor{red}{\ast}$ are the reconstructed radii. The blue dot $\textcolor{blue}{\circ}$ represents the accurate reconstruction, while the red start $\textcolor{red}{\ast}$ represents the inaccurate reconstruction.}}
  \label{r_vs_k_all}
\end{figure}

\begin{figure}[htbp]
  \centering
  \includegraphics[width = 6in]{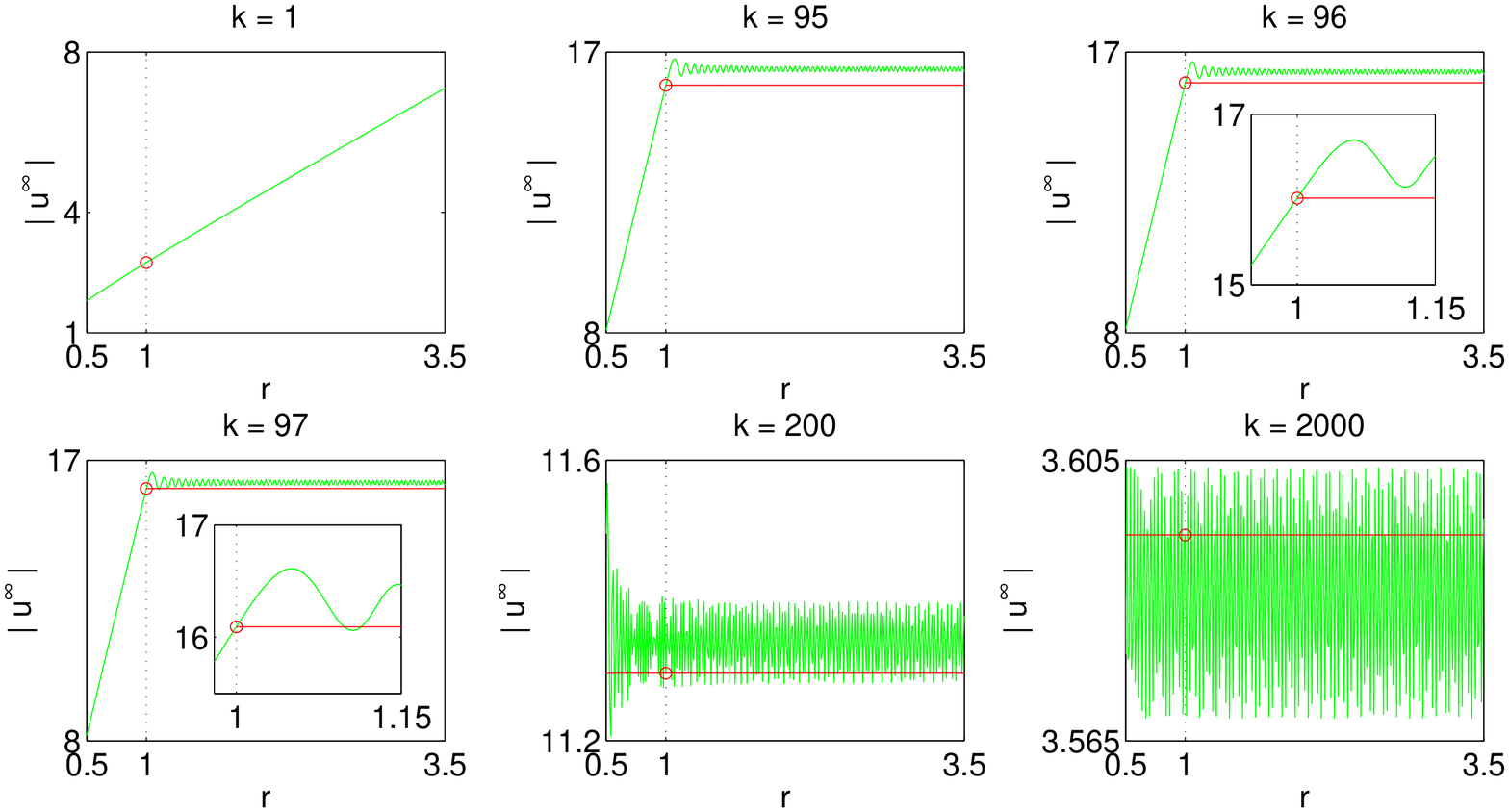}
  \caption{ Phaseless far-field pattern $|u^\infty(\hat{\mathbf{x}}_1)|$ vs radius $r$ for different wavenumbers $k$. }
  \label{u_vs_r_fixed_k}
\end{figure}

Having verified the accuracy of our inversion scheme at the ideal setting with a special sample of the observation noise, we now consider the inverse problem with a general sample of the observation noise at the noise coefficient $\sigma_{\eta} = 3\%$. In the second part, we estimate the obstacle parameters by setting $k=1$, $L=32$ and $M = 64$. We generate one sample of the observation noise, which is a matrix with $M \times L$ elements constructed by the formula \eqref{observation_noise}. The numerical approximations from the polluted observation data are exhibited in Figures \ref{z_r_noise} and \ref{hist_z_r_noise}. The mean of the numerical center and radius are $(1.0089, 0.2527)^\top$ and $0.1211$, respectively. The standard deviations of these parameters are $(0.0601, 0.0259)^\top, 0.0039$ and the relative errors are \tcr{$(0.89\%, 1.08\%)^\top, 0.92\%$}.

\begin{figure}[htbp]
  \centering
  \includegraphics[width = 3.8in]{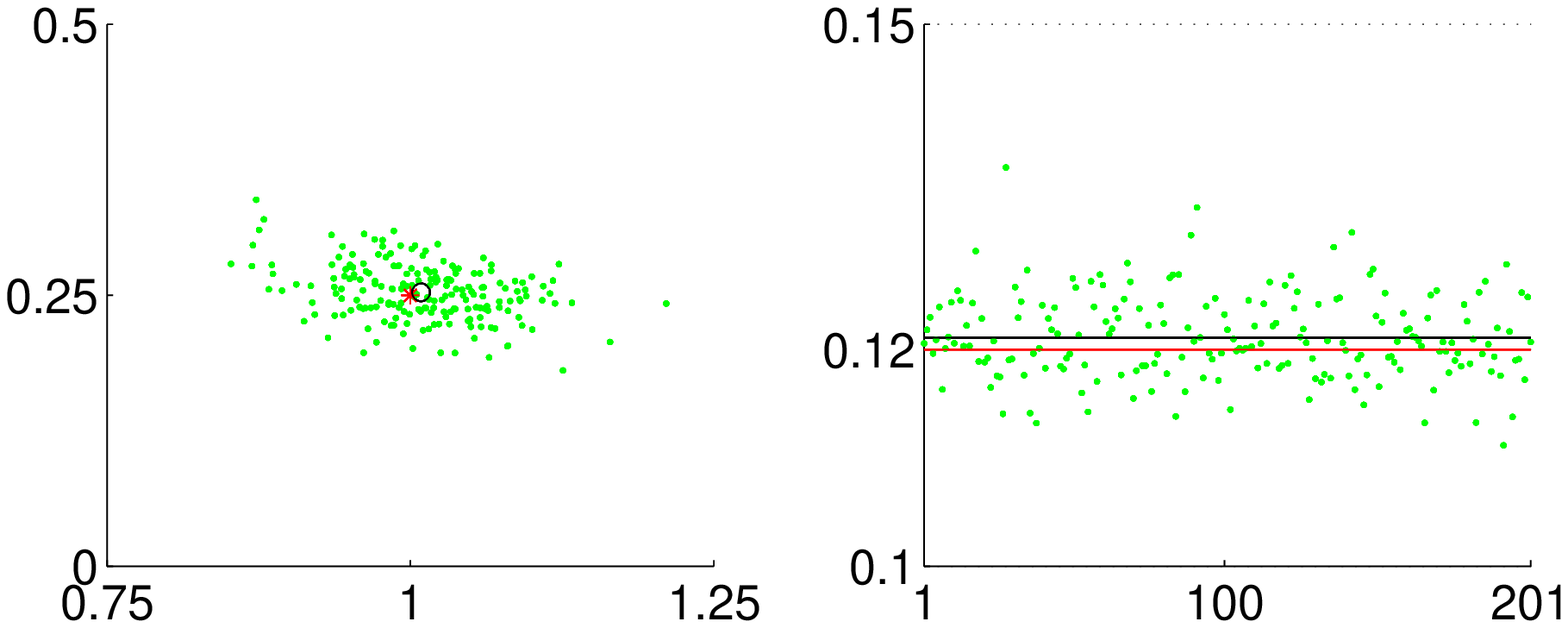}
  \caption{Reconstruction of the center (left) and radium (right) of a disk with one sample of observation noise at $\sigma_{\eta} = 3\%$. The red star $\textcolor{red}{\ast}$ (resp. line) is the accurate center (radius); the green dots $\textcolor{green}{\cdot}$ are the numerical reconstructions; the black circle $\circ$ (resp. line) is the mean of the centers and radii.  We choose $k=1$, $L=32$ and $M = 64$. }
  \label{z_r_noise}
\end{figure}

\begin{figure}[htbp]
  \centering
  \includegraphics[width = 6in]{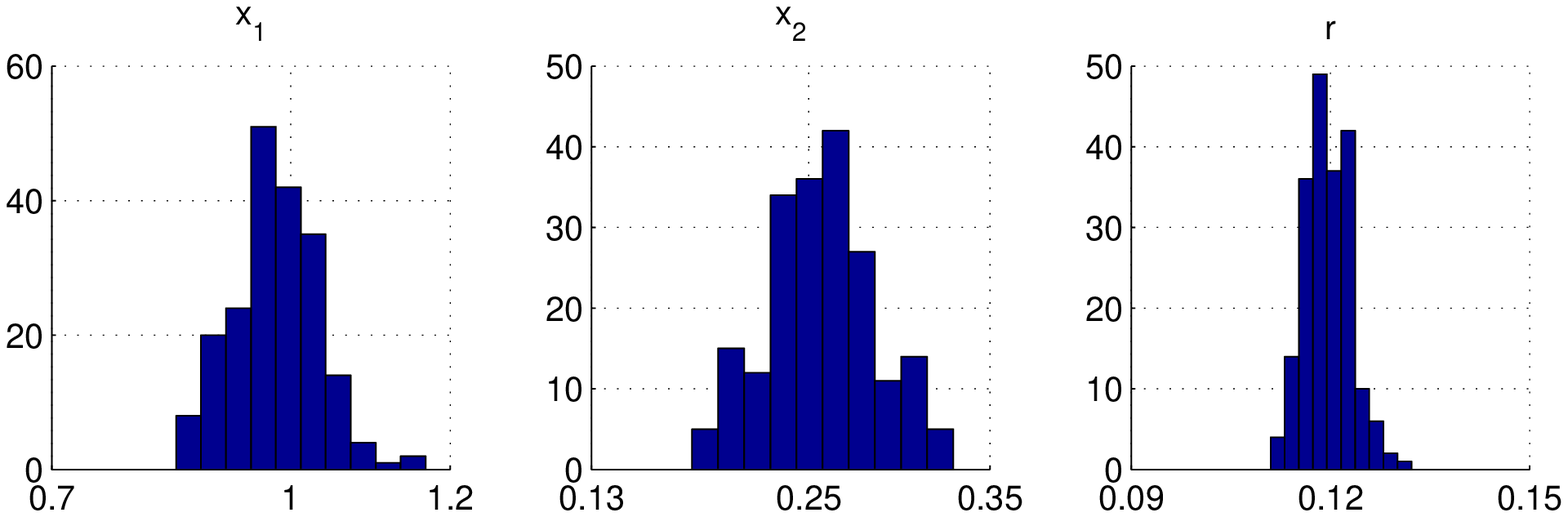}
  \caption{Histograms of the reconstructed parameters $x_{1}, x_{2}, r$ with one sample of the noisy data polluted at the level $\sigma_{\eta} = 3\%$.}
  \label{hist_z_r_noise}
\end{figure}

To demonstrate the robustness of the numerical scheme, we generate 1000 samples of the observation noise. For each sample \tcr{of the observation noise}, one can \tcr{gain a corresponding reconstruction} of the parameters $x_{1}, x_{2}, r$. Hence, we can perform statistical analysis over totally 1000 \tcr{reconstructions} of $x_{1}, x_{2}, r$. In our tests, we pollute the phaseless data at different levels $ \sigma_{\eta} = 3\%, 6\%, 9\% $ and exhibit the numerics in Table \ref{z_r_noise_3_table}, Figure \ref{z_r_noise_3} and Figure \ref{hist_z_r_noise_3}. From these reconstructed parameters we conclude that the mean and relative error are robust against the noise pollution, but the standard deviation is very sensitive to the noisy level. Further, the phaseless data with less noise give rise to a more reliable reconstruction result.

\begin{table}[htbp]
    \centering
    \caption{ The mean, standard deviation and relative error of $(x_{1}, x_{2}), r$ vs noise coefficient $\sigma_{\eta} $. }
    \label{z_r_noise_3_table}
    \begin{tabular}{c|c|c|c}
    \hline
    \  $\sigma_{\eta} $ & mean & standard deviation & relative error \\
    \hline
     $3\%$ & (0.9917, 0.2520), 0.1199 & (0.0077, 0.0039), 0.0005 & \tcr{(0.83\%, 0.79\%), \ 0.10\%} \\
    \hline
     $6\%$ & (0.9914, 0.2519), 0.1199 & (0.0137, 0.0065), 0.0009 & \tcr{(0.86\%, 0.77\%), \ 0.10\%} \\
    \hline
     $9\%$ & (0.9918, 0.2517), 0.1199 & (0.0199, 0.0095), 0.0013 & \tcr{(0.82\%, 0.69\%), \ 0.10\%} \\
    \hline
    \end{tabular}
\end{table}

\begin{figure}[htbp]
  \centering
  \includegraphics[width = 6in]{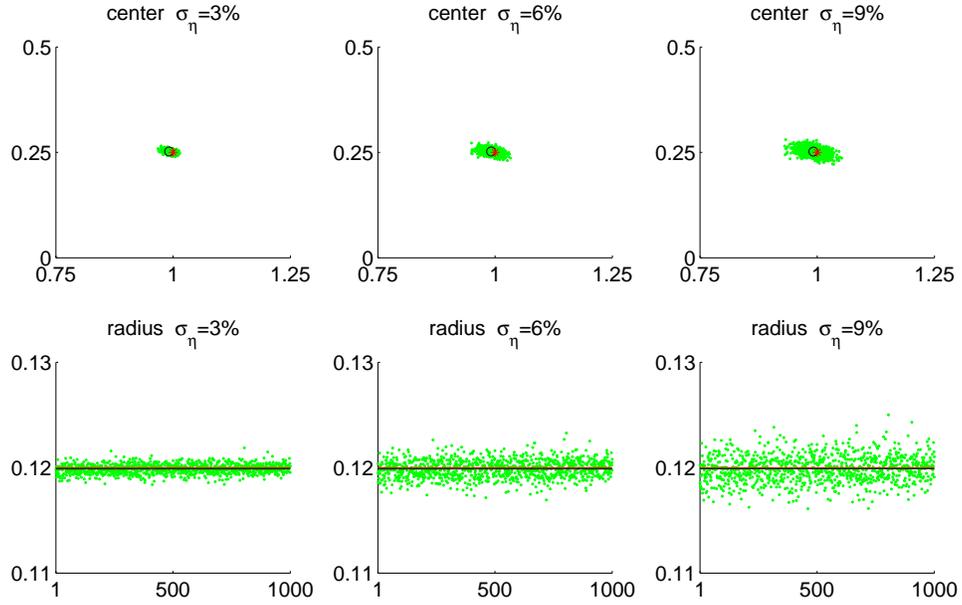}
  \caption{Reconstructions of the center (top) and radius (bottom) of a disk at different noise levels $\sigma_{\eta} = 3\%$ (left), $6\%$ (middle), $9\%$ (right). The red star $\textcolor{red}{\ast}$ (resp. line) is the accurate center (resp. radius), the green dots $\textcolor{green}{\cdot}$ are the numerical reconstructions with each sample of observation noise and the black $\circ$ (resp. line) is the mean of reconstructed centers (resp. radii) with 1000 samples of observation noise.}
  \label{z_r_noise_3}
\end{figure}

\begin{figure}[htbp]
  \centering
  \includegraphics[width = 5in]{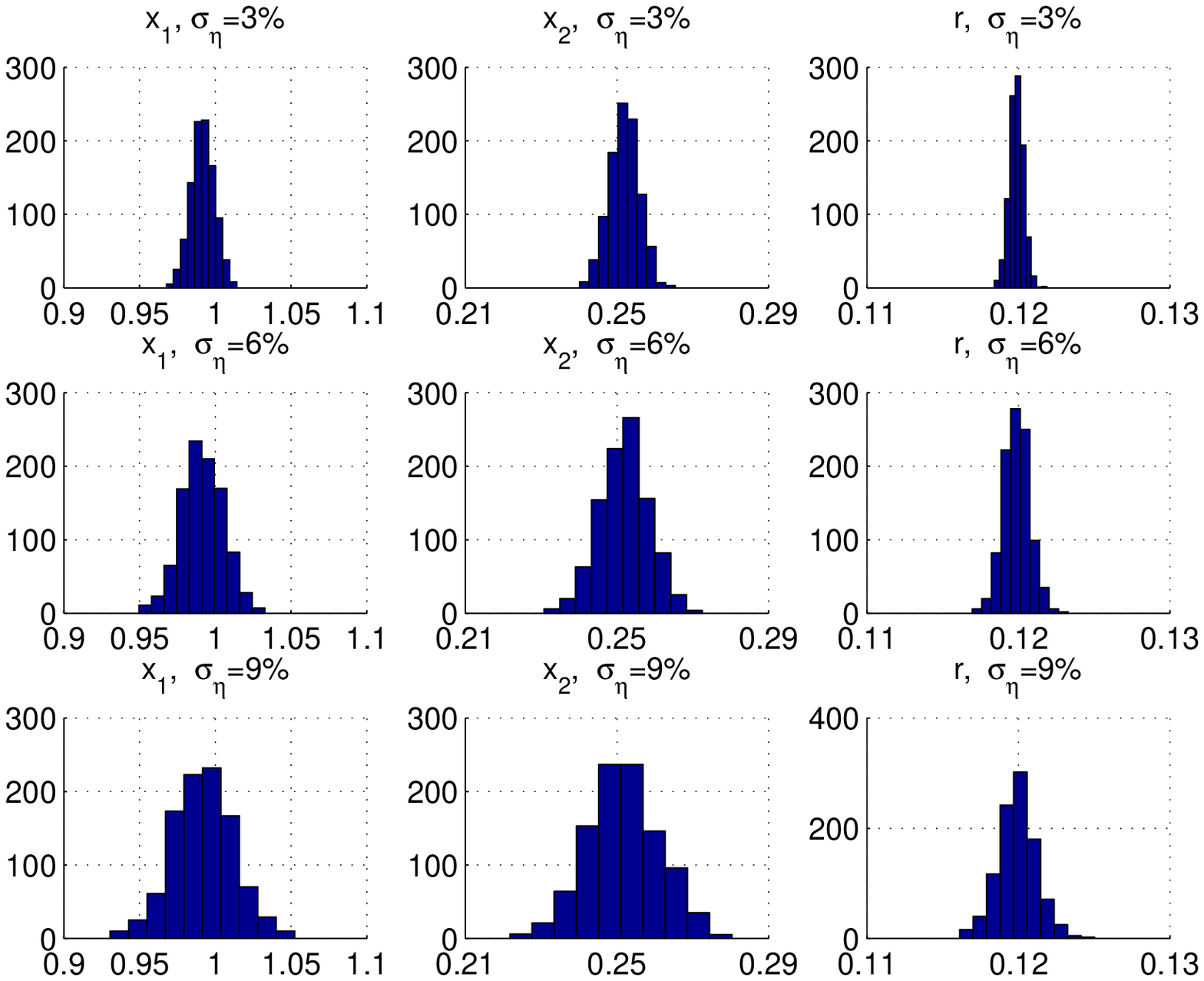}
  \caption{Histogram of 1000 numerical reconstructions of $x_{1}, x_{2}, r$ with each sample of observation noise at different noise levels $ \sigma_{\eta} = 3\%$ (top), $ \sigma_{\eta} = 6\% $ (middle), $ \sigma_{\eta} = 9\% $ (bottom).}
  \label{hist_z_r_noise_3}
\end{figure}

\subsection{Line cracks}
A crack or an open arc can be used to model the defects inside elastic and solid bodies such as bridge structures, aircraft engines and wings etc. Detection of such scatterers is important in safety and health assessment and is one of the fundamental topics in ultrasonic non-destructive testing. In this subsection, we want to recover a sound-soft crack of  line-segment-type with the starting point at $x=(x_{1}, x_{2})^{\top}\in \mathbb{R}^2$ and the ending point $y=(y_{1}, y_{2})^{\top}\in \mathbb{R}^2$. Hence, such line cracks can be characterized by $N=4$ parameters:
\begin{equation}
    \mathbf{Z}:= ( z_{1}, z_{2}, z_{3}, z_{4} )^{\top} = (x_{1}, x_{2}, y_{1}, y_{2} )^{\top}. \label{parameter_line}
\end{equation}

Unlike the scattering from disks, we do not have an analytical expression of the far-field pattern corresponding to a line crack. Below we describe the integral equation method to solve the forward scattering problem, following the numerical scheme of \cite{Kress1995_ScatterOpenArc} for general cracks. Denote by $\Gamma\subset\mathbb{R}^{2}$ an open arc of class $C^{3}$ in 2D, which can be parameterized as
\begin{equation}
    \Gamma = \left\{ z(s): s\in[-1, 1]  \right\}.
    \label{Gamma_z}
\end{equation}
Using the integral equation method, the solution $u^{sc}$ to the Helmholtz equation \eqref{Helmholtz_eq} in $\mathbb{R}^2\backslash \Gamma$ can be expressed as a single-layer potential (\cite{DColton_2013_InverseAcousticScattering})
\begin{equation}
    u^{sc}(x) = \int_{\Gamma} \Phi(x, y) \varphi(y) ds(y), \hspace{.2 cm} x\in \mathbb{R}^{2} \backslash \Gamma,
    \label{u_SingleLayer}
\end{equation}
where $\Phi(x, y)$ is the fundamental solution to the Helmholtz equation in two dimensions given by
\begin{equation}
    \Phi(x, y):= \frac{\mathrm{i}}{4}H^{(1)}_{0} (k|x-y|), \hspace{.2 cm} x, y\in \mathbb{R}^{2}, x\neq y.
    \label{Phi}
\end{equation}
Due to the Dirichlet boundary condition \eqref{DiriCond_Helmholtz_eq} on $\Gamma$, the unknown density function $\varphi$ is sought as a solution to the integral equation
\begin{equation}
    \int_{\Gamma} \Phi(x, y) \varphi(y) ds(y) = f(x), \hspace{.2 cm} x\in \Gamma,\qquad f=-u^{in}.
    \label{Phi_InterEqu}
\end{equation}
Once $\varphi$ is calculated from \eqref{Phi_InterEqu}, the far-field pattern could be expressed in the form
\begin{equation}
    u^{\infty}(\hat{x}) = \frac{ e^{\mathrm{i}\pi/4} }{ \sqrt{8\pi k} } \int_{\Gamma}e^{-ik \hat{x}\cdot y} \varphi(y) ds(y), \hspace{.2 cm} \hat{x}\in \mathbb{S}.
    \label{far_field_pattern_us}
\end{equation}
To describe the numerical scheme of \cite{Kress1995_ScatterOpenArc}, we first introduce two functions defined on $\mathbb{R}\times\mathbb{R}$ as follows
\begin{equation}
    H_{1}(t, \tau) :=
                \left\{
                        \begin{array}{ll}
                                J_{0}\left( k|z(\cos t)-z(\cos \tau)| \right) - 1, & t \neq \tau, \\
                                0, & t = \tau,
                        \end{array}
                \right. \label{Psi_InterEqu_H1}
\end{equation}
and
\begin{equation}
    H_{2}(t, \tau) :=
                \left\{
                        \begin{array}{ll}
                                \frac{\pi}{\mathrm{i}} H_{0}^{(1)} \left( k|z(\cos t)-z(\cos \tau)| \right)
                                -\left\{ 1 + H_{1}(t, \tau) \right\} \ln \left( \frac{4}{e^{2}}[\cos t-\cos \tau]^{2} \right) , & t \neq \tau, \\
                                \frac{\pi}{\mathrm{i} }+ 2C + 2 \ln \left\{ \frac{k e}{4} \left| z^{\prime}( \cos t) \right| \right\}, & t = \tau.
                        \end{array}
                \right. \label{Psi_InterEqu_H2}
\end{equation}
Here, $C \approx 0.577216$ is the Euler's constant. Then the integral equation \eqref{Phi_InterEqu} can be rephrased as
\begin{equation}
    \frac{1}{2\pi} \int_{0}^{2\pi} K(t, \tau) \psi(\tau) ds(\tau) = g(t), \hspace{.2 cm} g(t) := -2f( z( \cos t ) ),
    \label{Psi_InterEqu_ReWrite}
\end{equation}
for $t\in [0, 2\pi]$. Here,
\begin{eqnarray}
    K(t, \tau) &=& \left\{  1 + \sin^{2}\frac{t-\tau}{2} K_{1}(t, \tau) \right\}\ln\left( \frac{4}{e}\sin^{2}\frac{t-\tau}{2} \right) + \frac{ 1 }{ 2 }  H_{2}(t, \tau), \label{Psi_InterEqu_K}\\
    K_{1}(t, \tau) &=&
                \left\{
                        \begin{array}{ll}
                                \frac{ H_{1}(t, \tau) }{ \sin^{2}\left( (t-\tau)/2 \right) }, & t \neq \tau, \\
                                - k^{2} \sin^{2}\left(t\right)|z\prime(\cos t)|^{2}, & t = \tau,
                        \end{array}
                \right. \label{Psi_InterEqu_K1} \\
    \tcr{K_{2}(t, \tau)} &\tcr{=}& \tcr{\frac{ 1 }{ 2 }  H_{2}(t, \tau).} \label{Psi_InterEqu_K2}
\end{eqnarray}
The quadrature method \cite{RChapko_1993_QuadratureIntEqua} can be employed to discretize the integral equation \eqref{Psi_InterEqu_ReWrite},  based on the trigonometric interpolation with $2n$ equidistant nodal points $t_{j} := \frac{ j\pi }{ n }$, $j = 0, 1, \cdots, 2n-1$. Then the unknown solution $\psi$ to the integral equation \eqref{Psi_InterEqu_ReWrite} can be approximated by the $2n$ discrete nodal values $\left\{ \psi_{j} = \psi(t_{j}) \right\}_{j = 0}^{  2n-1 }$. Since $\psi_{k} = \psi_{2n-k}, k = 1, 2, \ldots, n-1$ with $\psi_{j} =|\sin (t_{j})|\left|z^{\prime}(\cos (t_{j}))\right| \varphi(z(\cos (t_{j})))$, it suffices to compute the $n+1$ discrete nodal values $\left\{ \psi_{j} \right\}_{j = 0}^{n}$ from the following $(n+1)\times(n+1)$ algebraic system
\begin{eqnarray}
    \sum_{j=0}^{2n-1} \psi_{j} \left\{ R_{|k-j|} + F_{ |k-j|}  K_{1} \left( t_{k}, t_{j} \right) + \frac{1}{2 n} K_{2} \left( t_{k}, t_{j} \right) \right\}
    = g \left( t_{k} \right), \quad k = 0, 1, \ldots, n,
    \label{linear_system}
\end{eqnarray}
with
\begin{eqnarray}\label{linear_system_R}
    R_{j} &:=& \frac{1}{2n} \left\{ c_{0} + 2 \sum_{m=1}^{n-1} c_{m} \cos \frac{m j \pi}{n} + (-1)^{j} c_{n} \right\},\; c_{m} := - \frac{1}{ \max\{1,|m|\} },\\ \label{linear_system_F}
    F_{j} &:=& \frac{1}{2 n}\left\{ \gamma_{0} + 2 \sum_{m=1}^{n-1} \gamma_{m} \cos \frac{m j \pi}{n} + (-1)^{j} \gamma_{n} \right\},\quad \gamma_{m} := \frac{1}{4} \left( 2 c_{m} - c_{m+1} - c_{m-1} \right).
\end{eqnarray}
Note that there are totally $n+1$ unknown discrete nodal values $\left\{ \psi_{j} \right\}_{j = 0}^{n}$ in \eqref{linear_system}, because $\psi_{j} = \psi_{n - |n-j|}$ for all $j = 0, 1, \cdots, 2n-1$. Now the far-field pattern can be approximated by
\begin{equation}\label{far_field_pattern_us_Rewrite}
        u^{\infty}(\hat{x}) =
        \frac{ e^{\mathrm{i}\pi/4} }{ \sqrt{8\pi k} } \int_{0}^{2 \pi} e^{-ik \hat{x}\cdot z(\cos \tau)} \psi(\tau) d\tau
          \approx \frac{ \pi e^{\mathrm{i}\pi/4} }{ n \sqrt{8\pi k} } \sum_{j=0}^{2n-1} e^{-ik \hat{x}\cdot z(\cos t_{j} )} \psi_{j}, \hspace{.2 cm} \hat{x}\in \mathbb{S}.
\end{equation}
If the right hand side of (\ref{linear_system}) (or (\ref{Phi_InterEqu})) is given by the incident wave \eqref{incident_wave}, we obtain the far-field pattern $u^{\infty}(\hat{\mathbf{x}}; \mathbf{Z}, \mathbf{d}_{0}, \mathbf{d}_{\ell}, k)$ where $\mathbf{Z}$ denotes the crack parameter \eqref{parameter_line}.

To set the parameters $\beta$ and $\Sigma_{pcn}$, we let $\Sigma_{pcn} = \mathbf{I}$ be the identity matrix, which is the same as the case of sound-soft disks. However, in this section the proposal variance coefficient $\beta$ is not a fixed number, but is taken as a random variable. This suggests that a random proposal variance is adopted to reconstruct line cracks. Then the proposal takes the form
\begin{equation}
    \mathbf{X} = \tcr{\mathbf{m}_{pr} + }( 1-\beta_{j}^2 )^{1/2} \tcr{(\mathbf{Z}_{j} - \mathbf{m}_{pr})} + \beta_{j} \omega, \hspace{.3cm} \omega \sim  \mathcal{N}( \mathbf{0}, \mathbf{I} ), \label{pCN_line}
\end{equation}
and the proposal variance coefficients $\beta_{j}$  need to be updated by the formulas \eqref{pCN_beta} and \eqref{update_beta}. It should be noted that, the MCMC method with a fixed proposal variance coefficient converges slowly or even does not converge after a large number of iterations, which is in contrast to the  efficient MCMC method for recovering disks. This  could partly be due to the number of reconstructed parameters, which is four in the line crack case while three for a disk. The trace of the iterations of MCMC (shown in Figure \ref{x_y_trace}) verifies the efficiency of the random proposal variance. In the first 10000 iteration steps, the trace converges fast but always drops into some fixed states, when the proposal variance coefficients $\beta_{j}$ are not appropriately updated. Numerics show that the trace can converge to and oscillate around the accurate state only after a large number of iterations.

As in the previous subsection, we set some computational parameters as follows:
\begin{itemize}
  \item The wave number $k$, the incident directions $\mathbf{d}_{\ell}, \ell=0,1,\cdots,L$, the observation directions $\hat{\mathbf{x}}_{m}, m=1,2,\cdots,M$, \tcr{$\sigma_{pr}$ of the prior distribution} and the observation pollution $\eta^{\ell}, \ell=1, \cdots, L$ are given as same as those for recovering disks;

  \item We choose \tcr{$L = M = 40$};

  \item \tcr{In the Algorithm \ref{pCN_Random_Proposal}, we choose $J_{1}=18000$, $J_{2}=5$, $J_{3}=401$;}

  \item The accurate obstacle is $\hat{\mathbf{Z}} = (\hat{x}_{1}, \hat{x}_{2}, \hat{y}_{1}, \hat{y}_{2})^{\top} = (2, 3, 4, 5)^{\top}$, that is, a line segment with the starting point $(2, 3)^\top$ and the ending point $(4, 5)^\top$;

  \item \tcr{We assume the initial guess is a line segment with the starting point $(0, 0)^\top$ and the ending point $(1, 1)^\top$. Then the mean of the prior distribution is $\mathbf{m}_{pr} = (0, 0, 1, 1)^{\top}$.}
\end{itemize}

Unfortunately we do not have the uniqueness result analogous to Theorem \ref{TH} for recovering cracks. A local uniqueness result for general cracks was proved in \cite{RKress_1995_FDiffFarFieldOperate} using a single far-field pattern \tcr{with information}. In the idea setting ($\omega^{\ell}_{m} = 0$ and $\sigma_{\eta} = 3\%$\tcr{, $\ell = 1, 2, \cdots, L$, $m = 1, 2, \cdots, M$}), the numerical approximations of the crack parameters are exhibited in Figure \ref{x_y_NoNoise}. The mean solutions of the starting and ending points are \tcr{$(2.0012,    3.0005)$, $(3.9990,    4.9992)$}, the standard deviations are \tcr{$(0.0139,    0.0164 )$, $(0.0129,    0.0170)$} and the relative errors are \tcr{$(0.06\%, 0.02\%)$, $(0.02\%, 0.02\%)$}.

\begin{figure}[htbp]
  \centering
  \includegraphics[width = 5in]{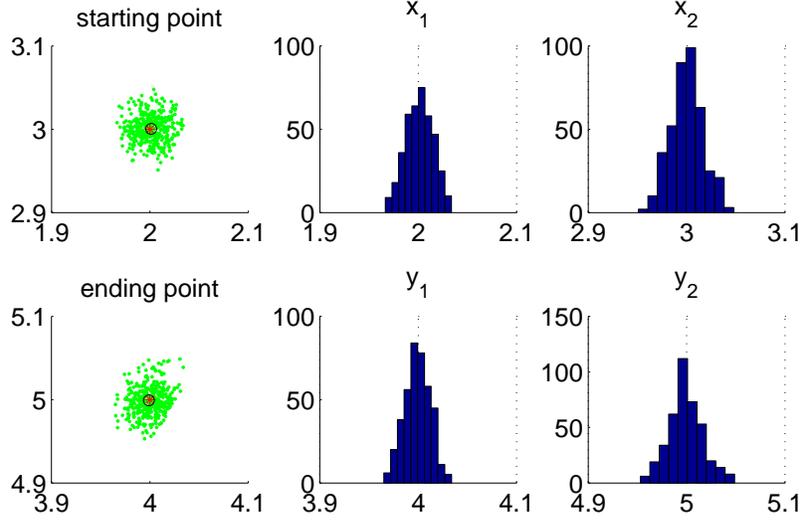}
  \caption{Histogram (center, right) and scatterer plot (left) of the reconstructions of starting point (top) and ending point (bottom) of a line crack with \tcr{the idea setting}. The red star $\textcolor{red}{\ast}$ denotes the accurate point, the green dots $\textcolor{green}{\cdot}$ are the numerical points, and the black $\circ$ is the mean of numerical points. }
  \label{x_y_NoNoise}
\end{figure}

Setting the noise coefficient $\sigma_{\eta} = 3\%$, we generate one general sample of the observation noise by the formula \eqref{observation_noise}, which takes the form of an $M \times L$ matrix. The numerical approximations from the polluted observations is exhibited in Figure \ref{x_y_scatter_hist}. The mean solutions of the numerical starting and ending points are \tcr{$(2.0052,    3.0031)$, {$(3.9958,    4.9986)$}, the standard deviations are \tcr{$(0.0135,    0.0155)$, $(0.0130,    0.0184)$} and the relative errors are \tcr{$(0.26\%, 0.10\%)$}, $(0.10\%, 0.03\%)$}. Besides, the trace of the iterations in MCMC are shown in Figure \ref{x_y_trace}.

\begin{figure}[htbp]
  \centering
  \includegraphics[width = 5in]{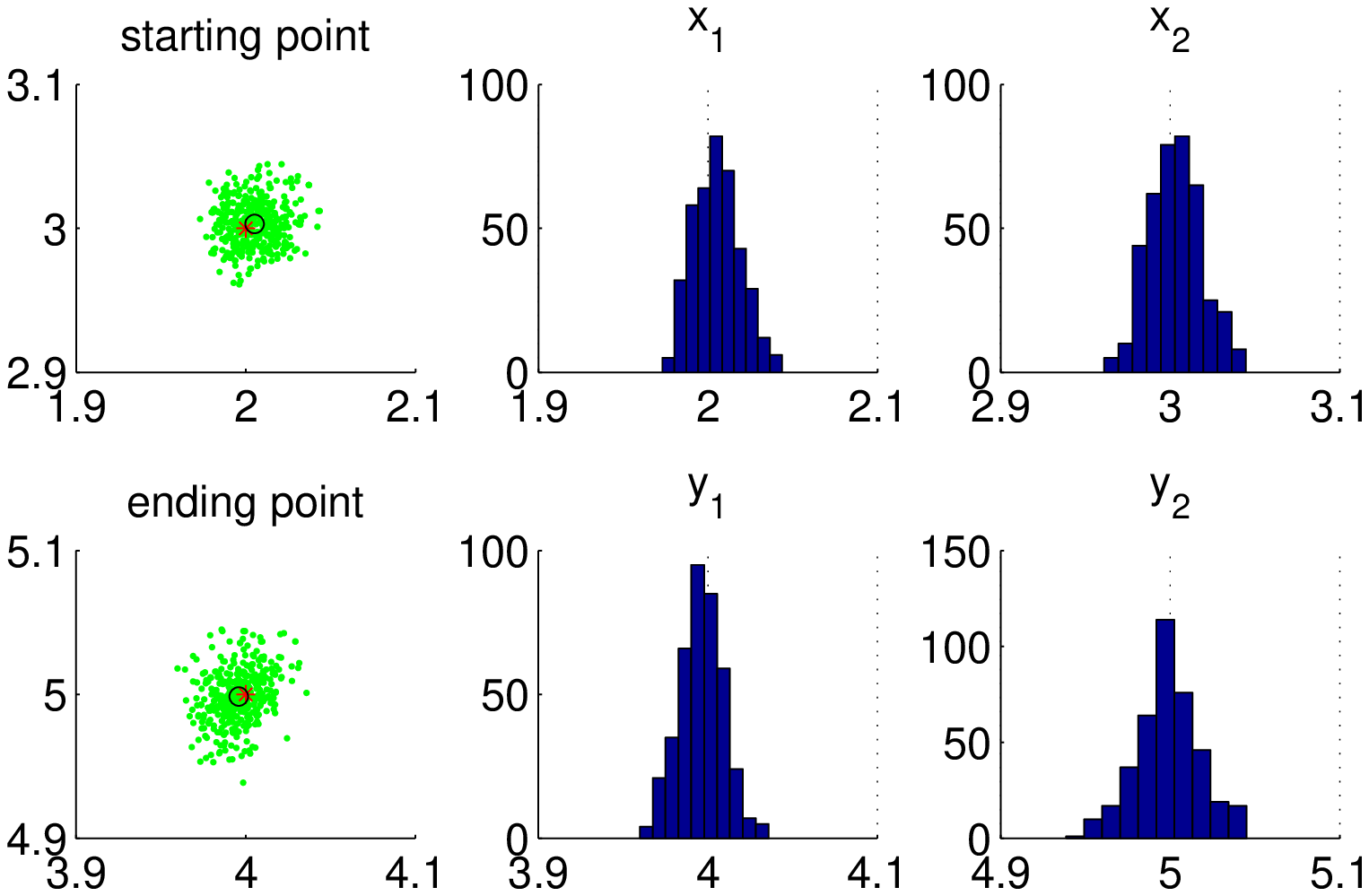}
  \caption{Histogram (center, right) and scatterer plot (left) of reconstructions of starting point (top) and ending point (bottom) of a line crack with one sample of observation noise at $\sigma_{\eta} = 3\%$. The red star $\textcolor{red}{\ast}$ denotes the accurate point, the green dots $\textcolor{green}{\cdot}$ are the numerical points, and the black $\circ$ is the mean of numerical points. }
  \label{x_y_scatter_hist}
\end{figure}

\begin{figure}[htbp]
  \centering
  \includegraphics[width = 6in]{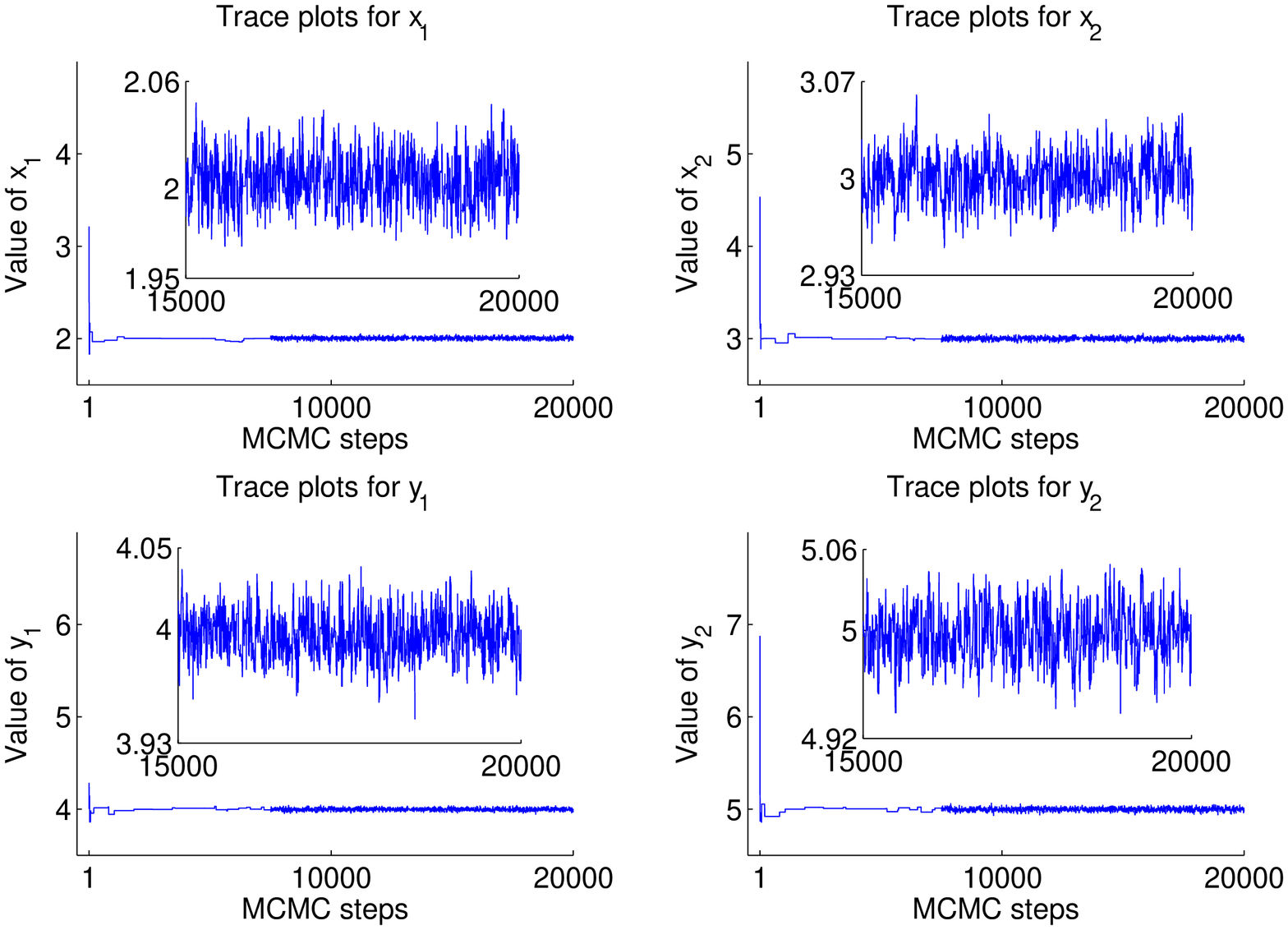}
  \caption{Trace of the iterations in the MCMC method with one sample of  the observation noise at $\sigma_{\eta} = 3\%$. The accurate line crack is $(\hat{x}_{1}, \hat{x}_{2}, \hat{y}_{1}, \hat{y}_{2})^{\top}= (2, 3, 4, 5)^{\top}$.}
  \label{x_y_trace}
\end{figure}

As done for recovering disks, we also demonstrate the robustness of the numerical scheme with 1000 samples of the observation noise at different levels $ \sigma_{\eta} = 3\%, 6\%, 9\% $. For each sample \tcr{of the observation noise}, one can \tcr{gain a corresponding reconstruction} of the parameters $x_{1}, x_{2}, y_{1}, y_{2}$. Hence, we can perform statistical analysis over totally 1000 \tcr{reconstructions} of $x_{1}, x_{2}, y_{1}, y_{2}$. The corresponding results are exhibited in Table \ref{x_y_noise_3_table}, Figure \ref{x_y_noise_3} and Figure \ref{x_y_noise_hist_3}. From these reconstructed parameters, we can draw almost the same conclusions as those for determining a sound-soft disk. The mean and relative error are robust against the noise pollution, but the standard deviation is very sensitive to the noisy level. It follows that the phaseless data with less noise give rise to a more reliable reconstruction result.

\begin{table}[htbp]
    \centering
    \caption{ The mean, standard deviation and relative error of $x_{1}, x_{2}, y_{1}, y_{2}$ vs noise coefficient $\sigma_{\eta} $. }
    \label{x_y_noise_3_table}
    \begin{tabular}{c|c|c|c}
    \hline
    \  $\sigma_{\eta} $ & mean & standard deviation \tcr{($10^{-3}$)} & relative error \tcr{($\%$)} \\
    \hline
     $3\%$ & \tcr{1.9994,    3.0055,    3.9992,   5.0029} & \tcr{3.5,    4.4,    3.4,    5.1} & \tcr{0.03, 0.18, 0.02, 0.06} \\
    \hline
     $6\%$ & \tcr{1.9989,    3.0056,    4.0013,    5.0133} & \tcr{6.2, 7.7, 6.1, 8.8} & \tcr{0.06, 0.19, 0.03, 0.27} \\
    \hline
     $9\%$ & \tcr{1.9936,    2.9994,    3.9968,    5.0081} & \tcr{8.6, 11.4, 8.9, 12.5} & \tcr{0.32, 0.02, 0.08, 0.16} \\
    \hline
    \end{tabular}
\end{table}

\begin{figure}[htbp]
  \centering
  \includegraphics[width = 6in]{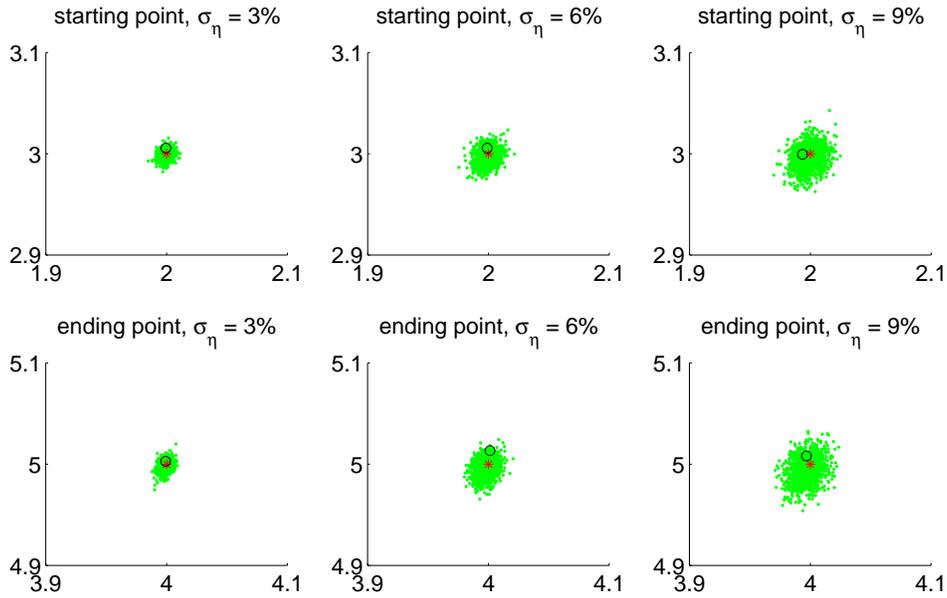}
  \caption{Reconstructions of starting point (top) and ending point (bottom) of the line crack at different noise levels $ \sigma_{\eta} = 3\%$ (top), $ \sigma_{\eta} = 6\% $ (middle), $ \sigma_{\eta} = 9\% $ (bottom). The red star $\textcolor{red}{\ast}$ denotes the accurate point, the green dots $\textcolor{green}{\cdot}$ are the numerical reconstructions with each sample of observation noise, and the black $\circ$ is the mean of numerical reconstructions with 1000 observation noises.}
  \label{x_y_noise_3}
\end{figure}

\begin{figure}[htbp]
  \centering
  \includegraphics[width = 6in]{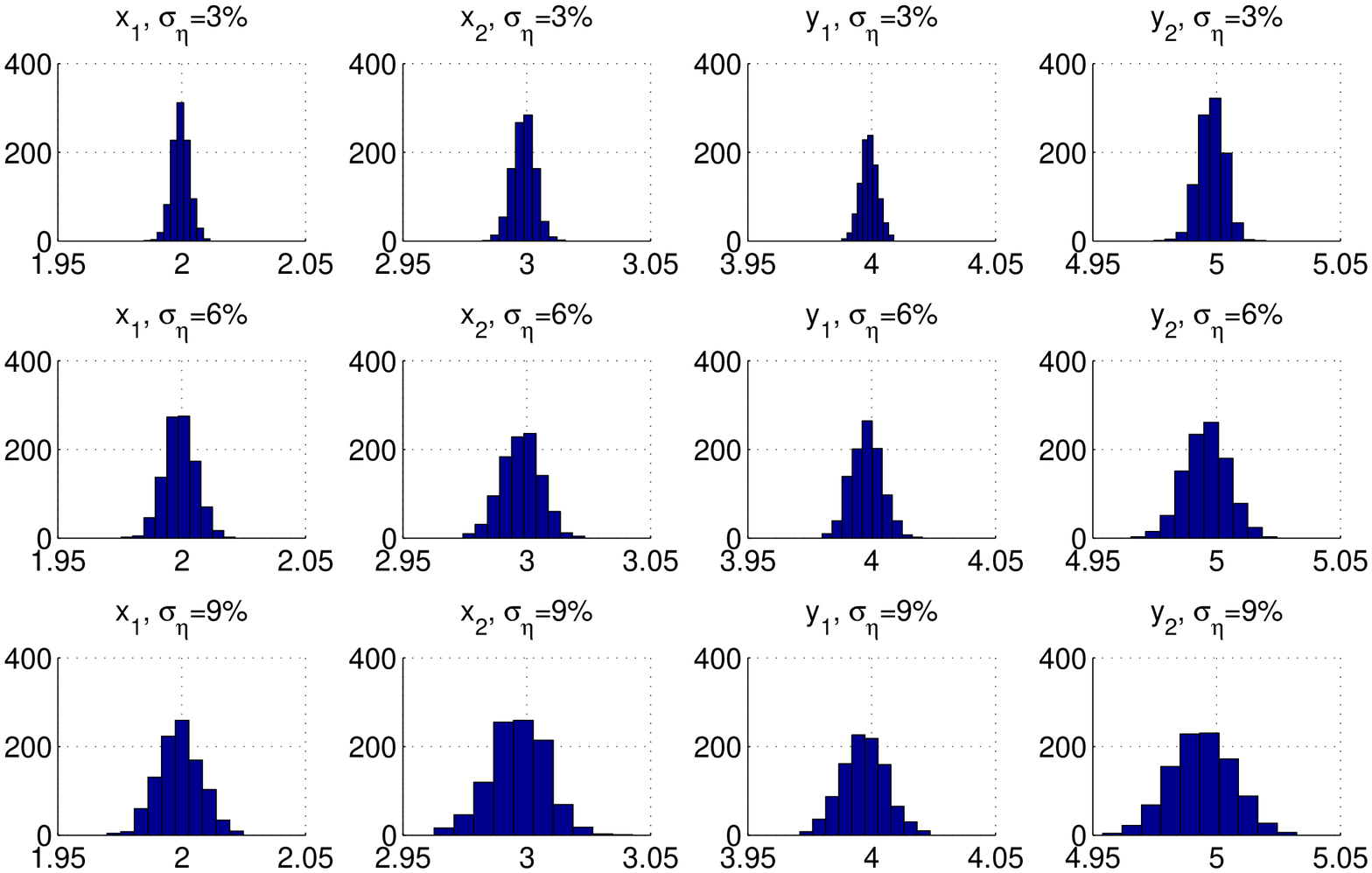}
  \caption{Histogram of 1000 numerical reconstructions of $x_{1}, x_{2}, y_{1}, y_{2}$ with each sample of the observation noise at the noise level $ \sigma_{\eta} = 3\%$ (top), $ \sigma_{\eta} = 6\% $ (middle), $ \sigma_{\eta} = 9\% $ (bottom).}
  \label{x_y_noise_hist_3}
\end{figure}

\subsection{\tcr{Kite-shaped obstacle}}

\tcr{In this subsection, we consider the following kite-shaped sound-soft obstacle $D_{1}$,
\begin{equation}
    \partial D_{1}
    = \left\{ x(t) = \big( - 0.65 + \cos t + 0.65 \cos (2t),
                           \ -3 + 1.5 \sin t \big)^{\top},
        \quad 0\leq t \leq 2\pi \right\},
    \label{kite_t}
\end{equation}
which is a benchmark acoustically impenetrable scatterer in inverse scattering problems. Obviously, the boundary of $D_{1}$ can be parameterized by six parameters
\begin{equation}
    \mathbf{Z}:= ( z_{1}, z_{2}, \cdots, z_{6} )^{\top}.
    \label{parameter_kite}
\end{equation}
Suppose that the exact parameters are given by $\hat{\mathbf{Z}}$ $= (\hat{z}_{1}$, $\hat{z}_{2}$, $\cdots$, $\hat{z}_{6})^{\top} = (-0.65, -3, 1, 0.65, 1.5, 0)^{\top}$.}
\tcr{To calculate numerical solutions of the forward problem, we adopt the MATLAB code given by \cite[Chapter 8]{GNakamura_RPotthast_2015_InverseModeling}.}
\tcr{Since there are six unknown parameters, the random proposal variance is adopted as same as in recovering line cracks.}
\tcr{As in the previous subsections, we set the computational settings as follows.}
\begin{itemize}
  \item \tcr{The wave number is $k=2$;}

  \item \tcr{The directions of plane incident waves are
    \begin{equation}
        \mathbf{d}_{\ell} =(\cos\theta_{\ell}, \  \sin\theta_{\ell})^{\top},
        \quad \theta_{\ell} = 2\pi \ell/(L+1),
        \quad \ell=0, 1,\cdots,L;
        \label{direction_plane_kite}
    \end{equation}}

  \item \tcr{The observation directions are
    \begin{equation}
        \hat{\mathbf{x}}_{m} = \big(\cos\theta_{m}, \ \sin\theta_{m} \big)^{\top},
        \quad \theta_{m} = -\pi + 2\pi (m-1)/M,
        \quad m=1,2,\cdots,M;
        \label{direction_observation_kite}
    \end{equation}}

    \item \tcr{Unless otherwise specified, we choose $L = M = 50$;}

    \item \tcr{In the Algorithm \ref{pCN_Random_Proposal}, we choose $J_{1}=180000$, $J_{2}=20$, $J_{3}=1001$;}

    \item \tcr{Since the initial guess is assumed to be a unit circle centered at the origin, the mean of the prior distribution $\pi_{pr}$ is $\mathbf{m}_{pr} = (0, 1, 0, 0, 1, 0)^{\top}$. In this example, we also assume $\sigma_{pr} = 1$;}

    \item \tcr{The observation pollution $\eta^{\ell}, \ell=1, \cdots, L$, are given as same as those for recovering disks.}

\end{itemize}

\tcr{Since the sixth exact obstacle parameter $\hat{z}_6$ is $0$, we can not use the relative error to evaluate the accuracy of numerical results in this example. Instead, the Hausdorff distance (HD) is chosen to compute the distance between  reconstructed and exact boundaries.
Recall that the Hausdorff distance between two obstacles  $\partial D_{2}$ and $\partial D_{3}$ is defined by
\begin{equation}
    d_{H}(\partial D_{2}, \partial D_{3})
    := \max\bigg\{ \sup_{x \in \partial D_{2}}
                   \inf_{y\in \partial D_{3}}
                     |x - y|, \quad
                   \sup_{y \in \partial D_{3}}
                   \inf_{x \in \partial D_{2}}
                   |y - x| \bigg\}.
    \label{Hausdorff_distance}
\end{equation}
}

\tcr{In the first part, we consider the ideal setting ($\omega^{\ell}_{m} = 0$ and $\sigma_{\eta} = 3\%$, $\ell = 1, 2, \cdots, L$, $m = 1, 2, \cdots, M$) of  observations to investigate the accuracy of our numerical method. As done for recovering disks, we discuss the accuracy of  numerical solutions for different choice of $L \ (M = L)$. In Table \ref{kite_plane_NoNoise_diff_NInOb_table} and Figure \ref{kite_plane_NoNoise_diff_NInOb}, we exhibit the numerical reconstructions and the Hausdorff distances (HD) between the numerical reconstructions and the exact boundaries. We find that the reconstructed parameters are getting more  accurate as the number of incident waves and observation directions becomes larger. The Hausdorff distance is less than $0.001$ if we choose $L = M = 50, 100$. The numerical solution with $L = M = 5$ is unreliable as illustrated in the Figure \ref{kite_plane_NoNoise_diff_NInOb}.}

\tcr{To show that this method is not sensitive to initial guess, we exhibit the initial guess, the exact boundary and numerical reconstruction in Figure \ref{kite_plane_NoNoise_NInOb50}, where we set $L = M = 50$. We can obtain an accurate numerical solution with the Hausdorff distance (HD)  being $ 1.30\times 10^{-4} $, even if the initial guess of the obstacle is separated from the exact one.}

\tcr{We draw the histogram of the selected 1001 sates with $L = M = 50$ in the Figure \ref{kite_plane_histogram_50}, which are used to construct the posterior density. From Figure \ref{kite_plane_histogram_50}, we conclude that the posterior density is a Gaussian distribution and the mean of the posterior density approaches the exact obstacle parameters.}

\begin{table}[htbp]
    \centering
    \caption{Reconstruction and Hausdorff distance (HD) vs $L \ (M=L)$.}
    \label{kite_plane_NoNoise_diff_NInOb_table}
    \begin{tabular}{c|c|c}
    \hline
      $L$ & Reconstruction & HD  \\
    \hline
      5 & -0.6441,    \ 0.6287,    \ 1.0013,    \ 0.6460,    \ 1.4934,   -0.0051 & 3.6185 \\
    \hline
      25 & -0.6488,   -3.0037,    \ 1.0047,   \ 0.6479,   \ 1.5030,   -0.0043 & 0.0021 \\
    \hline
       50 & -0.6499,   -3.0001,    \ 1.0000,   \ 0.6503,   \ 1.4998,   -0.0009 & $1.30\times 10^{-4}$ \\
    \hline
       100 & -0.6497,   -2.9995,   \ 1.0000,   \ 0.6508,   \ 1.4999,    \ 0.0006 & $6.98\times 10^{-5}$ \\
    \hline
    \end{tabular}
\end{table}

\begin{figure}[htbp]
  \centering
  \includegraphics[width = 6in]{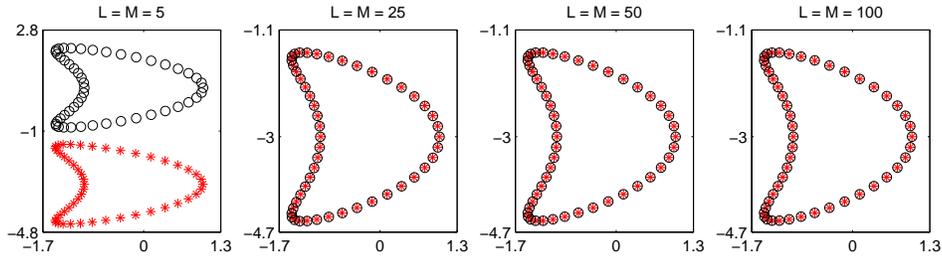}
  \caption{Reconstructions of a kite-shaped obstacle at the idea setting with different $L \ (M=L)$. The red star $\textcolor{red}{\ast}$ denotes the exact boundary and the black circle $\circ$ denotes the reconstructed boundary.}
  \label{kite_plane_NoNoise_diff_NInOb}
\end{figure}

\begin{figure}[htbp]
  \centering
  \includegraphics[width = 3in]{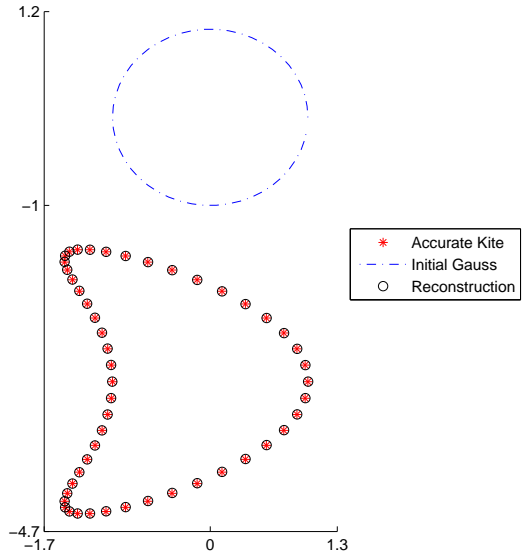}
  \caption{Reconstruction of a kite-shaped obstacle with $L=M=50$ at the idea setting.}
  \label{kite_plane_NoNoise_NInOb50}
\end{figure}

\begin{figure}[htbp]
  \centering
  \includegraphics[width = 4in]{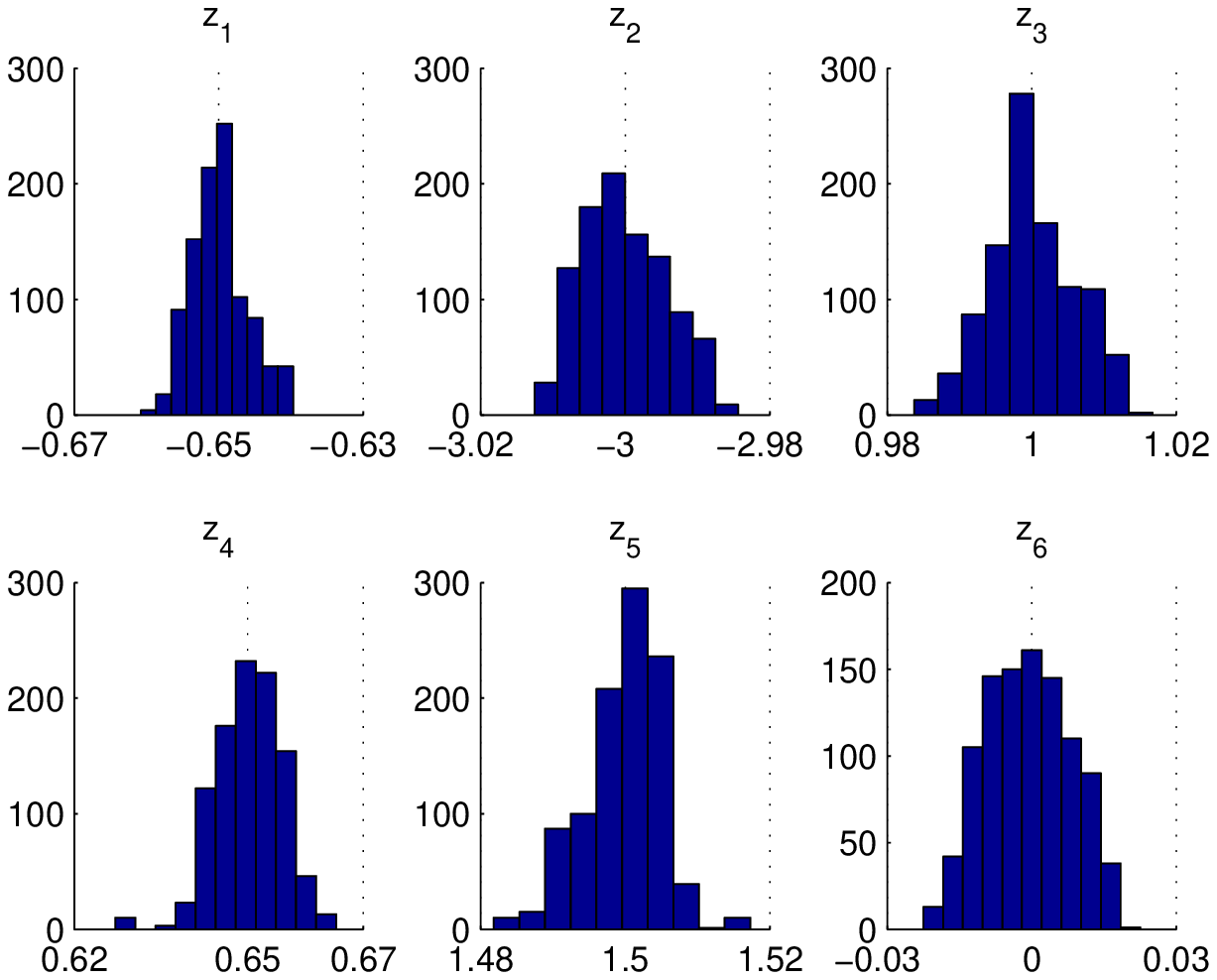}
  \caption{Histogram of the selected 1001 states with $L=M=50$ at the idea setting.}
  \label{kite_plane_histogram_50}
\end{figure}

\tcr{
Then we consider the practical setting with noise-polluted phaseless far-field data. As done for recovering disks and line cracks, we also demonstrate  robustness of the numerical scheme with 1000 samples of the observation noise at different noise levels $ \sigma_{\eta} = 3\%, 6\%, 9\%$ with $L = M = 50$. For every sample of the observation noise, one can gain a corresponding reconstruction of the parameters $z_{1}, z_{2}, \cdots, z_{6}$, and then gain the corresponding Hausdorff distance (HD) between the numerical reconstruction and the exact boundary. Then we can use standard statistical tools to analyze these 1000 reconstructions to discuss  robustness of our numerical scheme. In Table \ref{kite_plane_AddNoise_3_NInOb50_table}, we exhibit the mean of reconstructions, the mean and standard deviation (SD) of Hausdorff distances (HD) at different noise levels. In Figure \ref{kite_plane_AddNoise_3_NInOb50_curve}, we describe the 1000 reconstructions at different noise levels $\sigma_{\eta}$. In Figure \ref{kite_plane_AddNoise_3_NInOb50_1000_hist}, we show the histogram of these 1000 reconstructions, which correspond to the 1000 samples of the observation noise. We can find that the mean and standard deviation (SD) of Hausdorff distances (HD)  become larger as the noise level $\sigma_{\eta}$ is getting bigger. However, our inversion scheme is still robust against the noise pollution, since the mean and standard deviation (SD) of Hausdorff distances (HD) are very small as shown in Table \ref{kite_plane_AddNoise_3_NInOb50_table}. Further, the phaseless data with less noise give rise to a more reliable reconstruction result.}

\begin{table}[htbp]
    \centering
    \caption{Numerical solutions vs $\sigma_{\eta}$ with $L=M=50$.}
    \label{kite_plane_AddNoise_3_NInOb50_table}
    \begin{tabular}{c|c|c|c}
    \hline
    \  $\sigma_{\eta} $ & mean of reconstructions & mean of HD & SD of HD \\
    \hline
     $3\%$ & -0.6496,   -2.9988,    0.9988,    0.6506,    1.4988,    0.0010 &  0.0039 &  0.0035 \\
    \hline
     $6\%$ & -0.6493,   -2.9978,    0.9975,    0.6514,    1.4975,    0.0020 &  0.0062 &  0.0054 \\
    \hline
     $9\%$ & -0.6482,   -2.9960,    0.9959,    0.6517,    1.4962,    0.0031 &  0.0088 &  0.0073 \\
    \hline
    \end{tabular}
\end{table}

\begin{figure}[htbp]
  \centering
  \includegraphics[width = 6in]{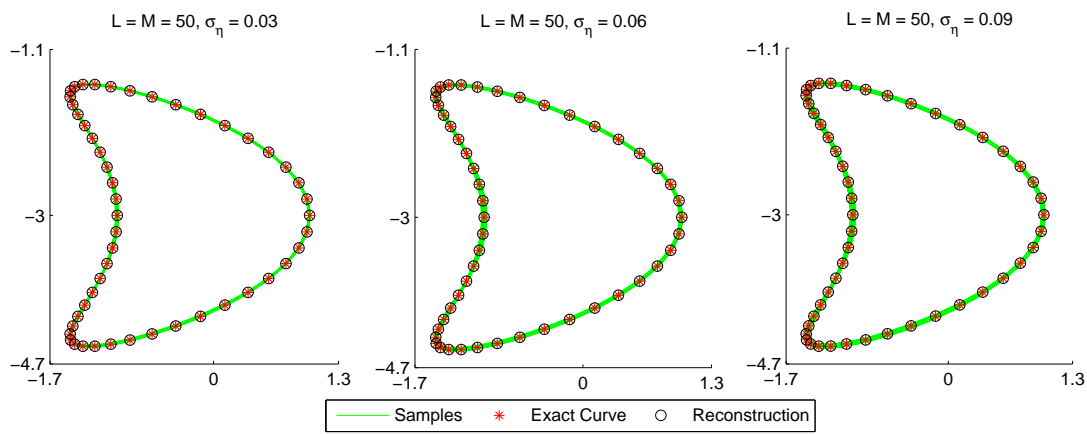}
  \caption{Reconstructions of a kite-shaped obstacle with $L=M=50$ at $\sigma_{\eta} = 3\%$ (left), $6\%$ (center), $9\%$ (right).}
  \label{kite_plane_AddNoise_3_NInOb50_curve}
\end{figure}

\begin{figure}[htbp]
  \centering
  \includegraphics[width = 5in]{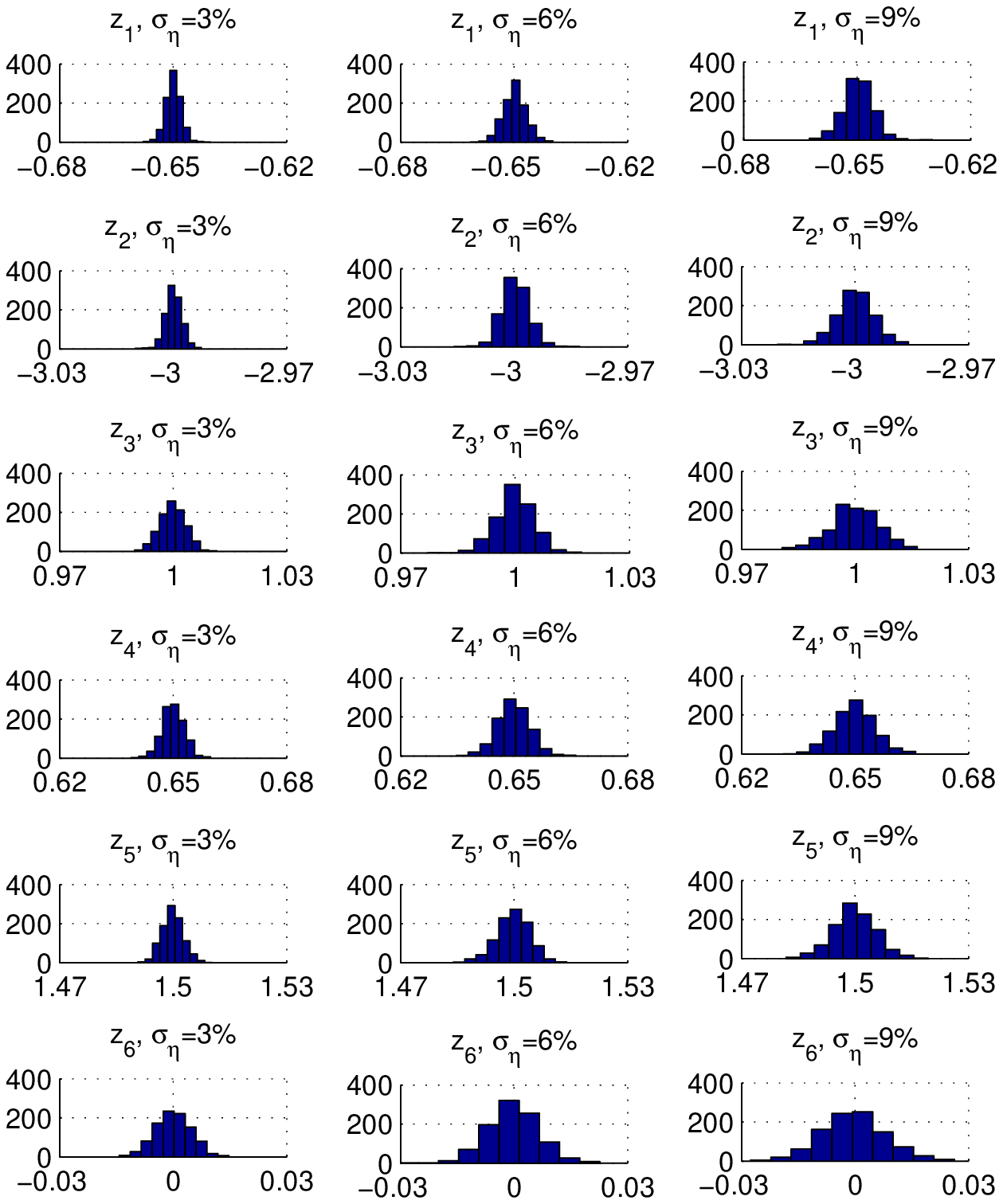}
  \caption{Histogram of these 1000 reconstructions of the six obstacle parameters of the kite-shaped obstacle with $L=M=50$ at $\sigma_{\eta} = 3\%$ (left), $6\%$ (center), $9\%$ (right).}
  \label{kite_plane_AddNoise_3_NInOb50_1000_hist}
\end{figure}

\section{Conclusion}
In this paper, we propose the Bayesian approach to inverse acoustic scattering from sound-soft disks\tcr{, line cracks and kite-shaped obstacles} with phaseless far-field data. Motivated by \cite{BZhang_2018_UniquenessScatterPhaseless}, the incoming waves are properly chosen in order to break the translational invariance of the far-field patten. Uniqueness of the inverse solution is proven for recovering a disk. When the Gaussian prior measure is given, we discuss \tcr{well posedness} of the posterior measure based on regularity properties of the deterministic direct scattering problem. Our numerics verify the efficiency of the preconditioned Crank-Nicolson algorithm with the random proposal variance. Further, increasing the number of incident and observation directions would lead to more accurate and reliable reconstructions. It is shown that the Bayesian method is robust for phaseless inverse scattering problems with respect to the observation noise. In this paper the obstacle boundary can be easily parameterized in a finite dimensional space. \tcr{Our future efforts will be devoted to recovering the shape and physical properties of more general acoustic obstacles with a large number of unknown parameters from phaseless far-field patterns. Noe that in this paper the number of unknown obstacle parameters is not larger than six, which has reduced the computational cost. For more complex scatterers, the increased computational cost of the Markov chain Monte Carlo method needs to be improved, for example, by combining the Gibbs sampling method and the stochastic surrogate model of the forward solver. Besides the idea of using superposition of two plane wave, one can also make use of a single spherical incident wave within the Bayesian framework. Research outcomes along these directions will be reported in our forthcoming publications.
}

%\bibliographystyle{plain}
%\bibliography{ref_Bayesian_Acoustic_phaseless}

\end{document}